\pdfoutput=1
\documentclass[leqno]{article}
\usepackage[normal]{phaine_style}
\usepackage[margin=1.5in]{geometry}
\usepackage{setspace}

\tikzcdset{
  cells={font=\everymath\expandafter{\the\everymath\displaystyle}},
}

\newcommand{\source}{\mathrm{s}}
\newcommand{\target}{\mathrm{t}}
\newcommand{\ac}{\mathrm{ac}} 

\newcommand{\Ani}{\categ{Ani}}

\newcommand{\sSpc}{\categ{sAni}}
\newcommand{\Nec}{\categ{Nec}}
\newcommand{\Necop}{\categ{Nec}^{\op}}

\newcommand{\Ar}{\mathrm{Ar}}

\newcommand{\too}{\longrightarrow}

\newcommand{\real}{\abs}

\newcommand{\Joints}{\mathrm{J}}
\newcommand{\act}{\mathrm{act}}
\newcommand{\Deltaact}{\DDelta^{\act}}

\newcommand{\BAut}{\mathrm{BAut}}

\newcommand{\Lan}{\mathrm{Lan}}
\renewcommand{\Ran}{\mathrm{Ran}}

\usepackage{caption}
\renewcommand{\Spc}{\Ani}

\DeclareSourcemap{
  \maps[datatype=bibtex]{
    \map{
      \step[fieldsource=note, final]
      \step[fieldset=addendum, origfieldval, final]
      \step[fieldset=note, null]
    }
  }
}

\title{\Large Fully faithful functors and pushouts of \categories}

\author{\normalsize Peter J. Haine \and\normalsize Maxime Ramzi \and\normalsize Jan Steinebrunner}
 
\date{\normalsize \today}

\begin{document}

\maketitle


\begin{abstract} 
	We study stability properties of fully faithful functors, and compute mapping anima in pushouts of \categories along fully faithful functors. 
	We provide applications of these calculations to pushouts along Dwyer functors and Reedy categories.  
\end{abstract}

\setcounter{tocdepth}{1}

\tableofcontents


\setcounter{section}{-1}

\section{Introduction}

While limits of categories and \categories are generally easy to understand from a computational perspective, colimits tend to be less tractable.
In principle, a colimit of \categories can be computed by unstraightening the functor and inverting all the cocartesian edges, but in general computing this localization is not feasible.

However, in special cases, colimits of \categories can be more approachable.
For example, colimits indexed by \groupoids or filtered colimits are rather straightforward. 
In this article, we focus on a special class of \emph{diagrams} as opposed to diagram \emph{shapes}; namely, we look at pushouts along fully faithful embeddings. 
In this case, we are able to compute all mapping anima in the pushout. 
The following is a combination of our main results, scattered as \Cref{thm:fully_faithful_functors_are_stable_under_pushout,cor:map3types,prop:map4}.
Given \acategory $ \Ccal $, we write $ \Ar(\Ccal) $ for the arrow \category of $ \Ccal $ and $ \abs{\Ccal} $ for the realization or classifying anima of $ \Ccal $.

\begin{theorem}\label{intro_thm:pushouts_along_fully_faithful_functors_and_their_mapping_anima}
   Consider a pushout square of \categories
	\begin{equation*}
        \begin{tikzcd}[column sep=2.5em]
            \Acal \arrow[r, "g"] \arrow[d, "f"', hooked] \arrow[dr, phantom, very near end, "\ulcorner", xshift=0.25em, yshift=-0.25em] & \Ccal \arrow[d, "\fbar"] \\
            \Bcal \arrow[r, "\gbar"'] & \Dcal \comma
        \end{tikzcd}
    \end{equation*} 
    where $ f $ is fully faithful. 
    In this case, $ \fbar $ is fully faithful. 
    Furthermore, for all $b\in \Bcal $ and $ c\in \Ccal$, we have equivalences: 
	\begin{enumerate}
	    \item $ \Map_{\Dcal}(\gbar(b),\fbar(c)) \equivalent |\Bcal_{b/}\crosslimits_{\Bcal} \Acal \crosslimits_{\Ccal} \Ccal_{/c}| $. 

	    \item $\Map_{\Dcal}(\fbar(c),\gbar(b)) \equivalent |\Ccal_{c/}\crosslimits_{\Ccal} \Acal \crosslimits_{\Bcal} \Bcal_{/b}|$.
	\end{enumerate}
    Finally, for all $b_0, b_1 \in \Bcal$ we have a pushout square
    \begin{equation*}
    	\begin{tikzcd}[column sep = large]
	        \abs{\Bcal_{b_0/} \crosslimits_{\Bcal} \Acal \crosslimits_{\Bcal} \Bcal_{/b_1}} \arrow[r] \arrow[d] \arrow[dr, phantom, very near end, "\ulcorner", xshift=0.25em, yshift=-0.25em] &
	        \Map_{\Bcal}(b_0, b_1) \arrow[d, "\gbar"] \\
	        \abs{\Bcal_{b_0/} \crosslimits_{\Bcal} \Acal \crosslimits_{\Ccal} \Ar(\Ccal) \crosslimits_{\Ccal} \Acal \crosslimits_{\Bcal} \Bcal_{/b_1}} \arrow[r] & 
	        \Map_{\Dcal}(\gbar(b_0), \gbar(b_1))
	    \end{tikzcd}
    \end{equation*}
    where the left-hand vertical map is induced by 
    \begin{align*}
        \Acal &\too \Acal \crosslimits_{\Ccal} \Ar(\Ccal) \crosslimits_{\Ccal} \Acal, \\ 
        a &\longmapsto (a, \id{g(a)}, a)
    \end{align*}
    and the horizontal maps are defined by composing.
\end{theorem}

For the first ``three types'' of mapping anima
\begin{equation*}
	\Map_{\Ccal}(\fbar(c),\fbar(c')) \comma \qquad \Map_{\Dcal}(\gbar(b),\fbar(c)) \comma \andeq \Map_{\Dcal}(\fbar(c),\gbar(b)) \comma
\end{equation*}
our methods our rather elementary.
They rely on the simple observation that applying $\PSh(-)$ to a pushout square yields a pullback square.
In fact, the fact that fully faithful functors are stable under pushout was already observed by Martini and Wolf \cite[Lemma 6.3.10]{MR4752519}, with the same method (they do not draw the mapping anima conclusions there). 
On the other hand, the ``fourth'' mapping anima $ \Map_{\Dcal}(\gbar(b_0),\gbar(b_1)) $ is more complicated, and, as the formula suggests, involves different methods.
Namely, the description of the fourth mapping anima relies on the explicit formula for Segalification via necklaces proven in \cite{arXiv:2301.0865}.

We provide various applications of \Cref{intro_thm:pushouts_along_fully_faithful_functors_and_their_mapping_anima} to pushouts along sieve inclusions, pushout products, and functors out of Reedy categories.
See \Cref{sec:examples_and_applications}.
We conclude the introduction by stating one sample application.

In studying the homotopy theory of $ 1 $-categories, Thomason introduced a class of fully faithful functors called \textit{Dwyer functors} \cite{MR591388}.
See \Cref{def:Dwyer}.
The main result of \cite{arXiv:2205.02353} is that (homotopy) pushouts of $ 1 $-categories along Dwyer functors remain $ 1 $-categories \cite[Theorem 1.6]{arXiv:2205.02353}.
We prove the following generalization of this result:

\begin{corollary}[{(\Cref{cor:pushouts_along_Dwyer_functors})}]
    Let $\Pcal\subset \Spc$ be a full subcategory of anima containing the empty set, and consider a pushout square of \categories 
    \begin{equation*}
        \begin{tikzcd}[column sep=2.5em]
            \Acal \arrow[r, "g"] \arrow[d, "f"', hooked] \arrow[dr, phantom, very near end, "\ulcorner", xshift=0.25em, yshift=-0.25em] & \Ccal \arrow[d, "\fbar"] \\
            \Bcal \arrow[r, "\gbar"'] & \Dcal \comma
        \end{tikzcd}
    \end{equation*}  
    where $ f $ is a Dwyer functor, and all mapping anima of $\Acal $, $ \Bcal $, and $ \Ccal $ lie in $ \Pcal $. 
    Then the mapping anima of $\Dcal$ also lie in $ \Pcal $.

    In particular, taking $ \Pcal $ to be the \category of $ (n-1) $-truncated anima, we deduce that the inclusion $ \incto{\Cat_{n}}{\Catinfty} $ of $ n $-categories into \categories preserves pushouts along Dwyer functors. 
\end{corollary}


\subsection{Related work}

In \cite{warn}, Wärn studies a related question, namely path anima in arbitrary pushouts of \groupoids, which is a special case of the general question of computing mapping anima in arbitrary pushouts of \categories.
Wärn's work is in a somewhat orthogonal direction from ours. 
The answer there involves more infinitary operations, which is to be expected as zigzag length in arbitrary localizations is unbounded. 

It would be very interesting to see whether the two approaches can be combined to compute general pushouts of \categories. 


\subsection{Linear overview}\label{intro-subsec:linear_overview}

In \Cref{sec:characterizations_of_fully_faithful_functors}, we record equivalent characterizations of full faithfulness which are helpful for our work but also worth knowing in general. 
In \Cref{sec:stability_properties_of_fully_faithful_functors} we study stability properties of fully faithful functors under categorical operations, one of which being pushouts along arbitrary functors. 
In \Cref{sec:mapping_anima_in_pushouts}, we compute all mapping anima in pushouts along fully faithful functors, modulo a key technical input for the ``fourth'' mapping anima, which we deal with in \Cref{sec:computing_pushouts_via_necklaces} using the theory of necklaces and results of \cite{arXiv:2301.0865}. 
Finally, in \Cref{sec:examples_and_applications}, we give a number of applications.
Among other things we generalize \cite[Theorem 1.6]{arXiv:2205.02353}, and we show that functors out of Reedy \categories admit a latching-matching description analogous to the $ 1 $-categorical situation.


\subsection{Notational conventions}

We write $ \Ani $ for the \category of anima (also referred to as spaces or \groupoids) and $ \Catinfty $ for the \category of \categories.
The inclusion $ \incto{\Ani}{\Catinfty} $ admits both a left and a right adjoint.
We denote the left and right adjoints by 
\begin{equation*}
	\real{-} \colon \fromto{\Catinfty}{\Ani} \andeq (-)^{\equivalent} \colon \fromto{\Catinfty}{\Ani} \comma
\end{equation*}
respectively.
We refer to $ \real{\Ccal} $ as the \defn{realization} or \defn{classifying anima} of $ \Ccal $ and refer to $ \Ccal^{\equivalent} $ as the \defn{(groupoid) core} of $ \Ccal $.


\subsection{Acknowledgments}

We thank David Nadler for asking questions that lead to us writing this note.
We also thank the members of the Copenhagen question seminar for inspiring the content of section \ref{subsec:Dwyer_functors}.

PH gratefully acknowledges support from the NSF Mathematical Sciences Postdoctoral Research Fellowship under Grant \#DMS-2102957. 
MR  was supported by the Danish National Research Foundation through the Copenhagen Centre for Geometry and Topology (DNRF151) while part of this work was conducted, and by the Deutsche Forschungsgemeinschaft (DFG, German Research Foundation) -- Project-ID 427320536 -- SFB 1442, as well as by Germany's Excellence Strategy EXC 2044 390685587, Mathematics Münster: Dynamics--Geometry--Structure during the finishing stages of writing.
JS was supported by the Independent Research Fund Denmark (grant no.~10.46540/3103-00099B) and the Danish National Research Foundation through the "Copenhagen Centre for Geometry and Topology" (grant no.~CPH-GEOTOP-DNRF151).


\section{Characterizations of fully faithful functors}\label{sec:characterizations_of_fully_faithful_functors}

In this section, we recall some basic characterizations of fully faithful functors.
We begin by recalling when adjoints are fully faithful.

\begin{lemma}[{\cite[Lemma 3.3.1]{ambidexterity}}]\label{lem:characterization_of_fully_faithful_adjoints}
	Let $ \adjto{L}{\Ccal}{\Dcal}{R} $ be an adjunction.
	The following are equivalent:
	\begin{enumerate}
		\item The left adjoint $ L $ is fully faithful.

		\item The unit $ \eta \colon \fromto{\id{\Ccal}}{RL} $ is an equivalence.

		\item The composite $ RL \colon \fromto{\Ccal}{\Ccal} $ is an equivalence of \categories (e.g., there exists an equivalence $ RL \equivalent \id{\Ccal} $).
	\end{enumerate}
\end{lemma}

\begin{lemma}\label{lem:triple_of_adjoints_with_fully_faithful_extremes}
    Let $L \leftadjoint f \leftadjoint R$ be a chain of adjunctions. 
    Then $L$ is fully faithful if and only if $R$ is fully faithful.
\end{lemma}

\begin{proof}
	By \Cref{lem:characterization_of_fully_faithful_adjoints}, $ L $ is fully faithful if and only if $ fL \equivalent \id{} $.
	By uniqueness of adjoints, this is equivalent to the requirement that $\id{} \equivalent fR$.
	Again applying \Cref{lem:characterization_of_fully_faithful_adjoints}, this is equivalent to saying that $ R $ is fully faithful.
\end{proof}

As a consequence, we deduce the following characterization of fully faithful functors in terms of presheaves.
	
\begin{lemma}\label{lem:characterization_of_fully_faithful_functors_via_Kan_extension}
	The following are equivalent for a functor of \categories $ f \colon \fromto{\Ccal}{\Dcal} $:
	\begin{enumerate}
		\item The functor $ f $ is fully faithful.

		\item The functor $ \flowershriek \colon \fromto{\PSh(\Ccal)}{\PSh(\Dcal)} $ given by left Kan extension along $ f $ is fully faithful.

		\item The functor $ \flowerstar \colon \fromto{\PSh(\Ccal)}{\PSh(\Dcal)} $ given by right Kan extension along $ f $ is fully faithful.
	\end{enumerate}
\end{lemma}

\begin{proof}
	To see that (2) $ \Rightarrow $ (1), note that by \HTT{Proposition}{5.2.6.3}, there is a commutative square 
	\begin{equation*}
			\begin{tikzcd}[sep=2.5em]
				\Ccal \arrow[r, "f"] \arrow[d, "\yo"', hooked] & \Dcal \arrow[d, "\yo", hooked]  \\ 
				 \PSh(\Ccal) \arrow[r, "f_!"'] &  \PSh(\Dcal)
			\end{tikzcd}
		\end{equation*}
	where $ \yo $ is the Yoneda embedding. 
	Since the Yoneda embedding is fully faithful, if $ f_! $ is fully faithful, then $ f $ is also fully faithful. 

	To see that (1) $ \Rightarrow $ (2), note that (the proof of) \HTT{Proposition}{4.3.2.17} shows that if $ f $ is fully faithful, then the unit for the $f_! \leftadjoint f^*$ adjunction is an equivalence.
	Hence $f_!$ is also fully faithful. 

	That (2) and (3) are equivalent follows from \Cref{lem:triple_of_adjoints_with_fully_faithful_extremes}.
\end{proof}
   



We conclude with the following Segal anima characterization of full faithfulness.

\begin{proposition}[{(see also \cite[Propositions 3.8.6 \& 3.8.7]{arXiv:2103.17141})}]\label{prop:Segal_anima_characterization_of_full_faithfulness}
	The following are equivalent for a functor of \categories $ f \colon \fromto{\Ccal}{\Dcal} $:
	\begin{enumerate}
		\item The functor $ f $ is fully faithful.

		\item The induced square
		\begin{equation*}
			\begin{tikzcd}[sep=3em]
				\Fun([1],\Ccal) \arrow[r, "f \of -"] \arrow[d, "{(\source,\target)}"'] & \Fun([1],\Dcal) \arrow[d, "{(\source,\target)}"]  \\ 
				\Ccal \cross \Ccal \arrow[r, "f \cross f"'] & \Dcal \cross \Dcal
			\end{tikzcd}
		\end{equation*}
		is a pullback square of \categories.

		\item The induced square of groupoid cores
		\begin{equation*}
			\begin{tikzcd}[sep=3em]
				\Fun([1],\Ccal)^{\equivalent} \arrow[r, "f \of -"] \arrow[d, "{(\source,\target)}"'] & \Fun([1],\Dcal)^{\equivalent} \arrow[d, "{(\source,\target)}"]  \\ 
				(\Ccal \cross \Ccal)^{\equivalent} \arrow[r, "f \cross f"'] & (\Dcal \cross \Dcal)^{\equivalent}
			\end{tikzcd}
		\end{equation*}
		is a pullback square of anima.
	\end{enumerate}
\end{proposition}

\begin{proof}
	That (1) is equivalent to (3) follows from inspecting the vertical fibers of the square in (3). 
	Furthermore, if $ f $ is fully faithful, then so is $f^{[1]}$ and thus in (2) both horizontal functors are fully faithful and we can check that it is a pullback square by noting that an arrow is the essential image of $f^{[1]}$ if and only if both its source and target are.
	Finally, (2) implies (3) by observing that the groupoid core functor $(-)^\simeq$ preserves pullbacks. 
\end{proof}


\section{Stability properties of fully faithful functors}\label{sec:stability_properties_of_fully_faithful_functors}

In this section, we prove our main results concerning stability properties of fully faithful functors under specific colimits. 

\begin{proposition}
	Let $ \Ical $ \acategory, $ \Ccal_{\bullet},\Dcal_{\bullet} \colon \fromto{\Ical}{\Catinfty} $ diagrams, and $ f_{\bullet} \colon \fromto{\Ccal_{\bullet}}{\Dcal_{\bullet}} $ a natural transformation.
	Assume that for each $ i \in \Ical $, the functor $ f_i $ is fully faithful.
	Then:
	\begin{enumerate}
		\item The induced functor $ \fromto{\lim_{i \in \Ical} \Ccal_i}{\lim_{i \in \Ical} \Dcal_i} $ is fully faithful.

		\item If $ \Ical $ is filtered, then the induced functor $ \fromto{\colim_{i \in \Ical} \Ccal_i}{\colim_{i \in \Ical} \Dcal_i} $ is fully faithful.
	\end{enumerate}
\end{proposition}

\begin{proof}
	First note that $ [1] $ is a compact object of $ \Catinfty $ and the groupoid core functor $ (-)^{\equivalent} $ preserves filtered colimits.
	Thus the functors $ \fromto{\Catinfty}{\Ani} $ given by
	\begin{equation*}
		\goesto{\Ccal}{\Fun([1],\Ccal)^{\equivalent}} \andeq \goesto{\Ccal}{(\Ccal \cross \Ccal)^{\equivalent}}
	\end{equation*}
	preserve both limits and filtered colimits.
	Since limits commute and filtered colimits of anima commute with pullbacks, both claims follow from item (3) of \Cref{prop:Segal_anima_characterization_of_full_faithfulness}.
\end{proof}

\noindent We now record counterexamples to (2) when $\Ical$ is not filtered.

\begin{example}
    Let $ f \colon X \to Y $ be an arbitrary map of anima. 
    By \kerodon{04QR}, $ f $ can be represented as a colimit over $\Deltaop$ of a cofibration of simplicial sets $ f_{\bullet} \colon X_{\bullet} \to Y_{\bullet} $. 
    Since $ f_{\bullet} $ is a levelwise injection, $ f_{\bullet} $ is levelwise fully faithful. 
    Since $ f $ was arbitrary, this proves that fully faithful maps are not closed under geometric realizations. 
\end{example}

\begin{example}
    Consider the map of spans
    \begin{equation*}
    	\begin{tikzcd}
    		\set{1 < 2} \arrow[d, equals] & \set{1} \arrow[l, hooked'] \arrow[r, hooked] \arrow[d, hooked] & \set{0 < 1 < 2} \arrow[d, equals] \\ 
    		\set{1 < 2} & \set{1 < 2} \arrow[l, hooked'] \arrow[r, hooked] & \set{0 < 1 < 2} \period
    	\end{tikzcd}
    \end{equation*}
    The pushout in $ \Catinfty $ of the top row is the poset with elements $ 0 $, $ 1 $, $ 2 $, and $ 2' $ and generating relations $ 0 < 1 $, $ 1 < 2 $, and $ 1 < 2' $.
    (This can for example be shown using \cref{cor:sieve}.)
    On the other hand, the pushout in $ \Catinfty $ of the bottom row is the poset $ \set{0 < 1 < 2} $.
    In particular, the induced functor on pushouts is not fully faithful.
\end{example}

\begin{theorem}[(fully faithful functors are stable under pushout)]\label{thm:fully_faithful_functors_are_stable_under_pushout}
	Given a pushout square of \categories
	\begin{equation}\label{eq:fully_faithful_functors_are_stable_under_pushout}
        \begin{tikzcd}[column sep=2.5em]
            \Acal \arrow[r, "g"] \arrow[d, "f"', hooked] \arrow[dr, phantom, very near end, "\ulcorner", xshift=0.25em, yshift=-0.25em] & \Ccal \arrow[d, "\fbar"] \\
            \Bcal \arrow[r, "\gbar"'] & \Dcal \comma
        \end{tikzcd}
    \end{equation} 
    if $ f $ is fully faithful, then $ \fbar $ is fully faithful.
\end{theorem}

\noindent To prove this, we start by studying properties of pullbacks. 

\begin{proposition}\label{prop:pullbacks_of_functors_with_fully_faithful_left_adjoints}
	Let 
	\begin{equation}\label{eq:pullbacks_of_functors_with_fully_faithful_left_adjoints}
		\begin{tikzcd}[column sep=2.5em]
			\Wcal \arrow[r,"\pbarupperstar"] \arrow[d,"\qbarupperstar"'] \arrow[dr, phantom, very near start, "\lrcorner", xshift=-0.25em, yshift=0.25em] & \Ycal\arrow[d,"\qupperstar"] \\
			\Xcal \arrow[r,"\pupperstar"'] & \Zcal
		\end{tikzcd}
	\end{equation}
	be a pullback square of \categories. 
	If $ \pupperstar \colon \Xcal \to \Zcal $ admits a fully faithful left adjoint, then $ \pbarupperstar \colon \Wcal\to \Ycal$ admits a fully faithful left adjoint. 
\end{proposition}

\begin{proof}
	Let $ \plowershriek $ denote the fully faithful left adjoint to $ \pupperstar $, and $\eta \colon \equivto{\id{\Zcal}}{\pupperstar \plowershriek}$ the unit equivalence. 
	By the universal property of the pullback, the functors $\id{\Ycal} \colon \Ycal\to \Ycal $ and $ \plowershriek \qupperstar \colon \Ycal\to \Xcal$, together with the equivalence
	\begin{equation*}
		\begin{tikzcd}
			\qupperstar \circ \id{\Ycal} = \id{\Zcal}\circ \qupperstar \arrow[r, "\eta \qupperstar", "\sim"'{yshift=0.25ex}] & \pupperstar\plowershriek \qupperstar 
		\end{tikzcd}
	\end{equation*}
	induce a functor $ \pbarlowershriek \colon \Ycal\to\Wcal $.

	We claim that $ \pbarlowershriek $ is a fully faithful left adjoint to $ \pbarupperstar$. 
	More precisely, we show that the natural equivalence $ \delta \colon \equivto{\id{\Ycal}}{\pbarupperstar \pbarlowershriek} $ that we get from the definition of $\pbarlowershriek$ is a unit transformation, from which the full faithfulness follows at once. 
	We need to show that given $y\in \Ycal $ and $w\in\Wcal$, the composition
	\begin{equation*}
		\begin{tikzcd}
			\Map_{\Wcal}(\pbarlowershriek(y), w) \arrow[r] & \Map_{\Ycal}(\pbarupperstar\pbarlowershriek(y), \pbarupperstar(w)) \arrow[r, "-\of\delta"] & \Map_{\Ycal}(y,\pbarupperstar(w)) 
		\end{tikzcd}
	\end{equation*} 
	is an equivalence.
	Because the second map in this composite is an equivalence by definition, it suffices to prove that the first map in the composite is an equivalence.

	Since \eqref{eq:pullbacks_of_functors_with_fully_faithful_left_adjoints} is a pullback square, the projections induce an equivalence
	\begin{equation*}
		\Map_{\Wcal}(\pbarlowershriek(y),w) \equivalence \Map_{\Xcal}(\qbarupperstar\pbarlowershriek(y),\qbarupperstar(w)) \crosslimits_{\Map_{\Zcal}(\pupperstar\qbarupperstar\pbarlowershriek(y),\pupperstar\qbarupperstar(w))} \Map_{\Ycal}(\pbarupperstar\pbarlowershriek(y),\pbarupperstar(w))
	\end{equation*} 
	Therefore, it suffices to show that the projection
	\begin{equation*}
		\Map_{\Xcal}(\qbarupperstar\pbarlowershriek(y),\qbarupperstar(w))\to \Map_{\Zcal}(\pupperstar\qbarupperstar\pbarlowershriek(y),\pupperstar\qbarupperstar(w))
	\end{equation*}
	is an equivalence. 
	Under the canonical equivalence $ \qbarupperstar \pbarlowershriek \simeq \plowershriek \qupperstar  $, this projection is identified with the projection
	\begin{equation*}
		\Map_{\Xcal}(\plowershriek \qupperstar(y), \qbarupperstar(w))\to \Map_{\Zcal}(\pupperstar\plowershriek \qupperstar(y), \pupperstar\qbarupperstar(w)) \period
	\end{equation*} 
	Composing this map with precomposition along $\eta_{\qupperstar(y)}$, we get an equivalence, by the definition of an adjunction. 
	Since $\eta_{\qupperstar(y)}$ is an equivalence, the projection is also an equivalence, as desired.
\end{proof}

\begin{remark}
    See \cite[Lemma 6.3.9]{arXiv:2111.14495} for a generalization of \Cref{prop:pullbacks_of_functors_with_fully_faithful_left_adjoints} to internal higher categories.
\end{remark}

Note that in \Cref{prop:pullbacks_of_functors_with_fully_faithful_left_adjoints}, we obtain an equivalence $ \plowershriek \qupperstar \equivalent \qbarupperstar \pbarlowershriek $ by design. 
Given a commutative square as in \eqref{eq:pullbacks_of_functors_with_fully_faithful_left_adjoints} where $ \pupperstar $ and $ \pbarupperstar $ have left adjoints $ \plowershriek $ and $ \pbarlowershriek $, there is also a canonical map
\begin{equation*}
	\fromto{\plowershriek \qupperstar}{\qbarupperstar \pbarlowershriek} \comma
\end{equation*}
called the \textit{Beck--Chevalley map} or \textit{exchange transformation}.
See \cites[\HAthm{Definition}{4.7.4.13}]{HA}{MR255631}{MR4655344}.
We point out to the reader that in the situation where $ \plowershriek $ is fully faithful (but where the square is not assumed to be a pullback), the existence of \emph{an} equivalence $ \plowershriek \qupperstar \equivalent \qbarupperstar \pbarlowershriek $ implies that specifically the Beck--Chevalley map is an equivalence. 
This digression is not needed in the proof of \Cref{thm:fully_faithful_functors_are_stable_under_pushout}, but is good to know. 

\begin{lemma}\label{lem:adjointability_for_fully_faithful_adjoints}
	Let
	\begin{equation*}
		\begin{tikzcd}[column sep=2.5em]
			\Wcal \arrow[r,"\pbarupperstar"] \arrow[d,"\qbarupperstar"'] \arrow[dr, phantom, very near start, "\lrcorner", xshift=-0.25em, yshift=0.25em] & \Ycal\arrow[d,"\qupperstar"] \\
			\Xcal \arrow[r,"\pupperstar"'] & \Zcal
		\end{tikzcd}
	\end{equation*}
	be a commutative square where $ \pupperstar $ and $ \pbarupperstar $ admit fully faithful left adjoints $ \plowershriek $ and $ \pbarlowershriek $, respectively. 
	If the functor $\qbarupperstar \pbarlowershriek$ lands in the essential image of $\plowershriek$ (e.g., there exists an equivalence $ \plowershriek \qupperstar \equivalent \qbarupperstar \pbarlowershriek $), then the Beck--Chevalley map $ \fromto{\plowershriek \qupperstar}{\qbarupperstar \pbarlowershriek} $ is an equivalence. 
\end{lemma}

\begin{proof}
    Fix $x\in\Xcal $ and $ y\in\Ycal$. 
    Mapping into $x$, the Beck--Chevalley map at $y$ corresponds, by definition, to the following composite
    \begin{align*}
    	\Map(\qbarupperstar \pbarlowershriek(y),x) &\to \Map(\pupperstar \qbarupperstar \pbarlowershriek(y),\pupperstar(x)) \\ 
    	&\simeq \Map(\qupperstar \pbarupperstar \pbarlowershriek(y),\pupperstar(x)) \\ 
    	&\to \Map(\qupperstar(y),\pupperstar(x)) \\
    	&\simeq \Map(\plowershriek \qupperstar(y),x) \period
    \end{align*}
	Since $\qbarupperstar \pbarlowershriek$ lands in the essential image of $\plowershriek$ by assumption, the first map in this composition is an equivalence. 
	Since $ \pbarlowershriek $ is fully faithful, the unit map $\id{\Ycal}\to \pbarupperstar \pbarlowershriek$ is an equivalence; hence the final map in the composition is also an equivalence. 
	Hence for any $ x\in\Xcal $, the whole composition is an equivalence, from which it follows that the Beck--Chevalley map is an equivalence. 
\end{proof}

\noindent We can summarize \Cref{prop:pullbacks_of_functors_with_fully_faithful_left_adjoints} as saying that, under the hypotheses, $ \pbarupperstar $ admits a fully faithful left adjoint, and that the square is horizontally left adjointable. 

We can now prove that fully faithful functors are stable under pushout:

\begin{proof}[Proof of \Cref{thm:fully_faithful_functors_are_stable_under_pushout}]
	Applying $\PSh(-)$ to the pushout square \eqref{eq:fully_faithful_functors_are_stable_under_pushout} yields a pullback square
	\begin{equation*}
		\begin{tikzcd}[column sep=2.5em]
			\PSh(\Dcal) \arrow[r, "\fbarupperstar"] \arrow[d, "\gbarupperstar"'] \arrow[dr, phantom, very near start, "\lrcorner", xshift=-0.25em, yshift=0.25em] & \PSh(\Ccal) \arrow[d,"\gupperstar"] \\
			\PSh(\Bcal) \arrow[r,"\fupperstar"'] & \PSh(\Acal)
		\end{tikzcd}
    \end{equation*} 
	of presheaf \categories and restriction functors.
	By \Cref{lem:characterization_of_fully_faithful_functors_via_Kan_extension}, it suffices to prove that the left adjoint of the restriction functor $ \fbarupperstar \colon \PSh(\Dcal)\to \PSh(\Ccal)$ is fully faithful.
	Equivalently, by uniqueness of adjoints, we need to show that $ \fbarupperstar $ admits a fully faithful left adjoint.

	Since $ f $ is fully faithful, \Cref{lem:characterization_of_fully_faithful_functors_via_Kan_extension} shows that $ \fupperstar \colon \PSh(\Bcal)\to \PSh(\Acal) $ admits a fully faithful left adjoint.
	\Cref{prop:pullbacks_of_functors_with_fully_faithful_left_adjoints} shows that $ \fbarupperstar $ also admits a fully faithful left adjoint, as desired.
\end{proof}

\begin{corollary}\label{cor:adjbl}
  	Let
	\begin{equation*}
        \begin{tikzcd}[column sep=2.5em]
            \Acal \arrow[r, "g"] \arrow[d, "f"', hooked] \arrow[dr, phantom, very near end, "\ulcorner", xshift=0.25em, yshift=-0.25em] & \Ccal \arrow[d, "\fbar", hooked] \\
            \Bcal \arrow[r, "\gbar"'] & \Dcal \comma
        \end{tikzcd}
    \end{equation*} 
    be pushout square of \categories where $ f $ is fully faithful.
    Then the induced square of presheaf \categories 
    \begin{equation*}
		\begin{tikzcd}[column sep=2.5em]
			\PSh(\Dcal) \arrow[r, "\fbarupperstar"] \arrow[d, "\gbarupperstar"'] \arrow[dr, phantom, very near start, "\lrcorner", xshift=-0.25em, yshift=0.25em] & \PSh(\Ccal) \arrow[d,"\gupperstar"] \\
			\PSh(\Bcal) \arrow[r,"\fupperstar"'] & \PSh(\Acal)
		\end{tikzcd}
    \end{equation*} 
    is horizontally left adjointable; that is, the Beck--Chevalley map $ \fromto{f_!g^*}{\gbar^*\fbar_!} $ is an equivalence. 
\end{corollary}

\begin{proof}
	Combine \Cref{thm:fully_faithful_functors_are_stable_under_pushout,lem:adjointability_for_fully_faithful_adjoints}.
\end{proof}

We conclude by noting that the fact that fully faithful functors are stable under pushout immediate implies some more stability properties:

\begin{proposition}\label{prop:properties_of_a_class_of_morphisms_stable_under_pushout}
	Let $ \Ccal $ be \acategory with pushouts, and let $ \Pcal $ be a collection of morphisms in $ \Ccal $ that contains all equivalences and is stable under composition and pushout.
	Let 
	\begin{equation*}
 		\begin{tikzcd}[sep=3em]
			X_0 \arrow[d, "x"'] & W_0 \arrow[d, "w" description] \arrow[r] \arrow[l] & Y_0 \arrow[d, "y"] \\
			X_1 & W_1 \arrow[l] \arrow[r] & Y_1
		\end{tikzcd}
 	\end{equation*}
 	be a commutative diagram in $ \Ccal $.
 	Assume one of the following:
	\begin{enumerate}
		\item The left-hand square is a pushout and $ y \in \Pcal $.

	    \item The morphisms $ x $, $ y $, and the codiagonal $ \nabla_w \colon \fromto{W_1 \coproduct^{W_0} W_1}{W_1} $ are in $ \Pcal $.
	\end{enumerate}
	Then the induced morphism
	\begin{equation*}
    	\fromto{X_0 \coproductlimits^{W_0} Y_0}{X_1 \coproductlimits^{W_1} Y_1}
    \end{equation*}
    is in $ \Pcal $.
\end{proposition}

\begin{proof}[Proof of \Cref{prop:properties_of_a_class_of_morphisms_stable_under_pushout}]
	For (1), for $ i = 0,1 $, let $ P_i $ denote the pushout $ X_i \coproduct^{W_i} Y_i $.
	Consider the commutative cube
	\begin{equation*}
      \begin{tikzcd}[column sep={8ex,between origins}, row sep={8ex,between origins}]
            W_0  \arrow[rr] \arrow[dd]  \arrow[dr, "w" description] & & Y_0 \arrow[dd]  \arrow[dr, "y"] \\
            & W_1 \arrow[rr, crossing over] & & Y_1 \arrow[dd]  \\
            X_0 \arrow[rr] \arrow[dr, "x"'] & & P_0 \arrow[dr] \\
            & X_1 \arrow[rr] \arrow[from=uu, crossing over] & & P_1 \period
      \end{tikzcd}
    \end{equation*}
    By definition, the front and back vertical faces are pushout squares.
    By assumption, the left-hand vertical face is also a pushout square.
    By the gluing lemma for pushouts, the right-hand vertical face is also a pushout.
    Since $ y \in \Pcal $ and the class $ \Pcal $ is stable under pushout, we deduce that the induced map on pushouts $ \fromto{P_0}{P_1} $ is also in $ \Pcal $.

    Now we prove (2). 
    First note that the natural morphism $ \fromto{X_0 \coproduct^{W_0} Y_0}{X_1 \coproduct^{W_1} Y_1} $ factors as a composite 
    \begin{equation*}
    	\begin{tikzcd}
    		X_0 \coproductlimits^{W_0} Y_0 \arrow[r] & X_1 \coproductlimits^{W_0} Y_1 \arrow[r] & X_1 \coproductlimits^{W_1} Y_1 \period
    	\end{tikzcd}
    \end{equation*}
    Since $ \Pcal $ is stable under composition, it suffices to prove that each of these morphisms is in $ \Pcal $.
    (That is, we need to show the claim in the special cases when $ w $ is the identity, and when $ x $ and $ y $ are identities.)
    
   	For the left-hand morphism, note that the morphism $ \fromto{X_0 \coproductlimits^{W_0} Y_0}{X_1 \coproductlimits^{W_0} Y_1} $ factors as a composite
   	\begin{equation*}
    	\begin{tikzcd}[sep=4.5em]
    		X_0 \coproductlimits^{W_0} Y_0 \arrow[r, "x \coproduct \id{Y_0}"] & X_1 \coproductlimits^{W_0} Y_0 \arrow[r, "\id{X_1} \coproduct y"] & X_1 \coproductlimits^{W_0} Y_1 \period
    	\end{tikzcd}
    \end{equation*}
    Since $ x,y \in \Pcal $ and $ \Pcal $ is stable under pushout, both of the above morphisms are in $ \Pcal $.

   	That the right-hand morphism is in $ \Pcal $ follows from the fact that the natural square
   	\begin{equation*}
   		\begin{tikzcd}
   			W_1 \coproductlimits^{W_0} W_1 \arrow[r, "\nabla_w"] \arrow[d] & W_1 \arrow[d] \\ 
   			X_1 \coproductlimits^{W_0} Y_1 \arrow[r] & X_1 \coproductlimits^{W_1} Y_1
   		\end{tikzcd}
   	\end{equation*}
   	is a pushout (a proof of which we'll see in \cref{eqn:fold-pushout}), the assumption that $ \nabla_w \in \Pcal $, and the assumption that $ \Pcal $ is stable under pushout.
\end{proof}

\begin{example}
	In light of \Cref{thm:fully_faithful_functors_are_stable_under_pushout}, \Cref{prop:properties_of_a_class_of_morphisms_stable_under_pushout} applies when $ \Ccal = \Catinfty $ and $ \Pcal $ is the class of fully faithful functors.
\end{example}


\section{Mapping anima in pushouts}\label{sec:mapping_anima_in_pushouts}

In this section, we use \Cref{thm:fully_faithful_functors_are_stable_under_pushout} and its proof to compute most of the mapping anima in a pushout of \categories where one of the legs is fully faithful. Given a pushout square:
\begin{equation*}
    \begin{tikzcd}[column sep=2.5em]
        \Acal \arrow[r, "g"] \arrow[d, "f"', hooked] \arrow[dr, phantom, very near end, "\ulcorner", xshift=0.25em, yshift=-0.25em] & \Ccal \arrow[d, "\fbar", hooked] \\
        \Bcal \arrow[r, "\gbar"'] & \Dcal \comma
    \end{tikzcd}
\end{equation*}  
where $ f $ is fully faithful, there are four ``types'' of mapping anima to compute: those from objects in the image of $\Bcal$ to objects in the image of $\Ccal$,  those from objects in the image of $\Ccal$ to objects in the image of $\Bcal$, those from objects in the image of $\Bcal$ to themselves, and finally those from objects in the image of $\Ccal$ to themselves. 

\Cref{thm:fully_faithful_functors_are_stable_under_pushout} concerns the last type, and all the ``mixed'' types can be computed from its proof. 
Finally, the mapping anima between objects in the image of $\Bcal$ are more difficult to access, and indeed, the formula we obtain for them is more complex. We address them last, as the relevant techniques are different. 

Let us, however, point out that in \cref{section:sieves}, we outline a simpler computation of this last type in the case where the inclusion $f \colon \incto{\Acal}{\Bcal} $ is a sieve. 


\subsection{The first three types of mapping anima}

The first main result is: 
\begin{corollary}\label{cor:map3types}
	Let
	\begin{equation*}
	    \begin{tikzcd}[column sep=2.5em]
	        \Acal \arrow[r, "g"] \arrow[d, "f"'] \arrow[dr, phantom, very near end, "\ulcorner", xshift=0.25em, yshift=-0.25em] & \Ccal \arrow[d, "\fbar"] \\
	        \Bcal \arrow[r, "\gbar"'] & \Dcal \comma
	    \end{tikzcd}
	\end{equation*} 
	be a pushout square of \categories where $ f $ is fully faithful.
	For all $b\in \Bcal $ and $ c,c'\in \Ccal$, we have natural equivalences: 
	\begin{enumerate}
	    \item $ \Map_{\Dcal}(\fbar(c),\fbar(c')) \simeq \Map_{\Ccal}(c,c') $.

	    \item $\Map_{\Dcal}(\gbar(b),\fbar(c)) \simeq |\Bcal_{b/}\crosslimits_{\Bcal} \Acal \crosslimits_{\Ccal} \Ccal_{/c}| $. 

	    \item $\Map_{\Dcal}(\fbar(c),\gbar(b)) \simeq |\Ccal_{c/}\crosslimits_{\Ccal} \Acal \crosslimits_{\Bcal} \Bcal_{/b}|$.
	\end{enumerate}
\end{corollary}

\begin{proof}
	Item (1) is simply full faithfulness of $\fbar$, so it follows from \Cref{thm:fully_faithful_functors_are_stable_under_pushout}. 
	
    Now let us prove (2).
    We note that since left Kan extension preserves representable functors,
    \begin{equation*}
    	\Map_{\Dcal}(\gbar(b),\fbar(c)) 
        = \left(\gbar^*\Map_{\Dcal}(-,\fbar(c))\right)(b)
        = \left(\gbar^*\fbar_! \Map_{\Ccal}(-,c)\right)(b) \period
    \end{equation*} 
	Thus, using \Cref{cor:adjbl} and the pointwise formula for left Kan extensions, we find
	\begin{align*}
		\Map_{\Dcal}(\gbar(b),\fbar(c)) &= f_!g^*\Map_{\Ccal}(-,c)(b) \\
		&= \colim_{(\Acal\crosslimits_{\Bcal} \Bcal_{b/})^{\op}} \Map_{\Ccal}(g(a),c) \period
	\end{align*}	
	Now, in terms of right fibrations, the functor
	\begin{equation*}
		\Acal^{\op}\to\Spc \comma \quad a\mapsto \Map_{\Ccal}(g(a),c)
	\end{equation*}
	is represented by the right fibration $\Acal\crosslimits_{\Ccal} \Ccal_{/c}$.
	Hence the same functor, restricted to $\Acal\crosslimits_{\Bcal} \Bcal_{b/}$ is represented by the pullback right fibration, i.e., $\Bcal_{b/}\crosslimits_{\Bcal} \Acal\crosslimits_{\Ccal} \Ccal_{/c}$. 
	Thus its colimit is indeed the realization of this \category. 

	Item (3) follows from (2) by dualizing.
	The functor $(-)^{\op} \colon \fromto{\Catinfty}{\Catinfty} $ is an equivalence of \categories, and hence preserves pushouts, and it also preserves full faithfulness. 
	Thus, by (2) applied to the opposite of the original square, we find
	\begin{align*}
		\Map_{\Dcal}(\fbar(c),\gbar(b)) &= \Map_{\Dcal^{\op}}(\gbar^{\op}(b),\fbar^{\op}(c)) \\
		&= \real{\Bcal^{\op}_{b/}\crosslimits_{\Bcal^{\op}} \Acal^{\op} \crosslimits_{\Ccal^{\op}} \Ccal^{\op}_{/c}} \\
		&= \real{\Ccal_{c/} \cross_{\Ccal} \Acal \cross_{\Bcal} \Bcal_{/b}} \comma
	\end{align*}
	as desired. 
\end{proof}

Using this formula we can detect when a square of fully faithful functors is a pushout square.

\begin{corollary}\label{cor:ff-and-ff}
     A square of \categories
     \begin{equation*}
        \begin{tikzcd}[column sep=2.5em]
            \Acal \arrow[r, "g", hooked] \arrow[d, "f"', hooked] & \Ccal \arrow[d, "\fbar", hooked] \\
            \Bcal \arrow[r, "\gbar"', hooked] & \Dcal \comma
        \end{tikzcd}
    \end{equation*}  
    where all functors are fully faithful is a pushout square if and only if the following conditions are satisfied:
    \begin{enumerate}
        \item\label{it:2ff-surjective} The functors $\fbar$ and $\gbar$ are jointly surjective.

        \item\label{it:2ff-ff} For all $b \in \Bcal \setminus \Acal$ and $c \in \Ccal\setminus \Acal$
        the maps
        \begin{equation*}
            |\Bcal_{b/} \crosslimits_\Bcal \Acal \crosslimits_\Ccal \Ccal_{/c}| 
            \longrightarrow \Map_\Dcal(\gbar(b), \fbar(c))
            \andeq
            |\Ccal_{c/} \crosslimits_\Ccal \Acal \crosslimits_\Bcal \Bcal_{/b}| 
            \longrightarrow \Map_\Dcal(\fbar(c), \gbar(b))
        \end{equation*}
        are equivalences.
    \end{enumerate}
\end{corollary}

\begin{proof}
    Let $ \Pcal $ denote the pushout - we obtain a functor $F\colon \fromto{\Pcal}{\Dcal} $. Since the maps $\Bcal\to \Pcal$ and $\Ccal\to \Pcal$ are jointly surjective, we see that $ F $ is essentially surjective if and only if (1) holds.

 Furthermore, $F$ is fully faithful on the image of $\Bcal$ and on the image of $\Ccal$. Hence $ F $ is fully faithful if and only if $ F $ induces equivalences
    \begin{equation*}
        \Map_\Pcal(b, c) \too \Map_\Dcal(b, c) 
        \andeq
        \Map_\Pcal(c, b) \too \Map_\Dcal(c, b) 
    \end{equation*}
    for all $b \in \Bcal \setminus \Acal$ and $c \in \Ccal \setminus \Acal$.
    By \cref{cor:map3types}, this is exactly equivalent to condition (2).
\end{proof}


\subsection{The fourth mapping anima}\label{section:fourth}

The goal of this subsection is to compute the missing ``fourth'' mapping anima that was not yet described in \cref{cor:map3types}.
This requires \cref{prop:fold-pushout} as an input, which we prove in the next section using different techniques.

\begin{proposition}\label{prop:map4}
   Let  
   \begin{equation*}
        \begin{tikzcd}[column sep=2.5em]
            \Acal \arrow[r, "g"] \arrow[d, "f"', hooked] \arrow[dr, phantom, very near end, "\ulcorner", xshift=0.25em, yshift=-0.25em] & \Ccal \arrow[d, "\fbar", hooked] \\
            \Bcal \arrow[r, "\gbar"'] & \Dcal \comma
        \end{tikzcd}
    \end{equation*}  
    be a pushout square of  \categories where $ f $ is fully faithful.
   	For all $b_0, b_1 \in \Bcal$, there is a pushout square of anima
    \begin{equation*}
		\begin{tikzcd}[column sep = large]
		    \abs{\Bcal_{b_0/} \crosslimits_{\Bcal} \Acal \crosslimits_{\Bcal} \Bcal_{/b_1}} \arrow[r] \arrow[d] &
		    \Map_{\Bcal}(b_0, b_1) \arrow[d, "\gbar"] \\
		    \abs{\Bcal_{b_0/} \crosslimits_{\Bcal} \Acal \crosslimits_{\Ccal} \Ar(\Ccal) \crosslimits_{\Ccal} \Acal \crosslimits_{\Bcal} \Bcal_{/b_1}} \arrow[r] & 
		    \Map_{\Dcal}(\gbar(b_0), \gbar(b_1)) \period
		\end{tikzcd}
    \end{equation*}
    Here, the left-hand vertical map is induced by 
    \begin{align*}
        \Acal &\too \Acal \crosslimits_{\Ccal} \Ar(\Ccal) \crosslimits_{\Ccal} \Acal, \\ 
        a &\longmapsto (a, \id{g(a)}, a)
    \end{align*}
    and the horizontal maps are defined by composing.
\end{proposition}

\begin{remark}
    If we have $b_0 \in \Acal$, then the top horizontal map in \cref{prop:map4} is an equivalence.
    In this case we can also rewrite the bottom left corner as 
    \begin{equation*}
    	|\Ccal_{g(b_0)/} \crosslimits_{\Ccal} \Acal \crosslimits_{\Bcal} \Bcal_{/b_1}| \comma
    \end{equation*}
    which recovers the formulas for the 2nd and 3rd mapping anima, so that the bottom map is an equivalence, thus \cref{prop:map4} in this case follows from \cref{cor:map3types}.
    The same argument applies if $b_1 \in \Acal$, so we could without loss of generality assume that $b_0, b_1 \not\in \Acal$ for the proof of \cref{prop:map4}, but this wouldn't lead to any simplifications.
\end{remark}

We now prove \cref{prop:map4} with a forward reference to the next section, where we use \textit{necklaces} to establish a certain pushout square of mapping anima.

\begin{proof}[Proof of \cref{prop:map4}]
	Fix $b_0, b_1 \in \Bcal$. We interpret these as objects $b_0, b_1 \in \Bcal \coproductlimits^{\Acal} \Bcal$ where $b_0$ lives in the left and $b_1$ in the right copy of $\Bcal \hookrightarrow \Bcal \coproductlimits^{\Acal} \Bcal$, respectively.
	In the next section we prove \cref{prop:fold-pushout}, which tells us that there is a pushout diagram of anima
	\begin{equation*}
	    \begin{tikzcd}
	        \Map_{\Bcal \coproductlimits^{\Acal} \Bcal}(b_0, b_1) \arrow[r, "\nabla_\Bcal"] \arrow[d] &
	        \Map_\Bcal(b_0, b_1) \arrow[d] \\
	        \Map_{\Dcal \coproductlimits^{\Ccal} \Dcal}(g(b_0), g(b_1)) \arrow[r, "\nabla_\Dcal"'] &
	        \Map_\Dcal(g(b_0), g(b_1))
	    \end{tikzcd}
	\end{equation*}
	where the horizontal maps come from the fold functors $\nabla \colon \Bcal \coproductlimits^{\Acal} \Bcal \to \Bcal$ and $\nabla \colon \Dcal \coproductlimits^{\Ccal} \Dcal \to \Dcal$.

	Using our computation of the second and third mapping anima in \cref{cor:map3types}, we can rewrite the left two terms as
	\begin{align*}
	    \Map_{\Bcal \coproductlimits^{\Acal} \Bcal}(b_0, b_1) &\simeq |\Bcal_{b_0/} \crosslimits_{\Bcal} \Acal \crosslimits_{\Bcal} \Bcal_{/b_1}| \\ 
	    \shortintertext{and}
	    \Map_{\Dcal \coproductlimits^{\Ccal} \Dcal}(\gbar(b_0), \gbar(b_1)) &\simeq |\Dcal_{\gbar(b_0)/} \crosslimits_{\Dcal} \Ccal \crosslimits_{\Dcal} \Dcal_{/\gbar(b_1)}| \period
	\end{align*}
	While the first one is exactly what we need, the second one might not seem as useful because it uses the \category $\Dcal$, which we are trying to compute.
	What we will show is that our understanding of the second and third mapping anima suffices to describe the relevant fiber products of slices of $\Dcal$.
	
	To start, consider the functor
	\begin{align*}
	    F\colon \Bcal_{b_0/} \crosslimits_\Bcal \Acal \crosslimits_\Ccal \Ar(\Ccal)
	    & \too \Dcal_{\gbar(b_0)/} \crosslimits_\Dcal \Ccal \\
	    (\phi\colon b_0 \to f(a), \psi\colon g(a) \to c) & \longmapsto (\fbar(\psi) \circ \gbar(\phi) \colon \gbar(b_0) \to \fbar(c))
	\end{align*}
	Then $ F $ is a map from a cocartesian fibration over $\Ccal$ (indeed a free cocartesian fibration) to a left fibration over $\Ccal$.
	By \cite[Lemma A.1.8]{Haugseng2020}, $ F $ is itself a cocartesian fibration.
	(Condition (4) of \cite[Lemma A.1.8]{Haugseng2020} is satisfied as the cocartesian edges of $F_c$ are precisely the equivalences and thus are preserved by cocartesian transport along any $c \to c'$.)
	Taking fibers over $c \in \Ccal$, the functor $ F $ induces the functor
	\begin{equation*}
	    F_c \colon \Bcal_{b_0/} \crosslimits_\Bcal \Acal \crosslimits_\Ccal \Ccal_{/c} 
	    \too \Map_\Dcal(\gbar(b_0), \fbar(c)) \comma
	\end{equation*}
	which is a localization by \cref{cor:map3types}.
	Therefore $ F $ is a cocartesian fibration with weakly contractible fibers.
	By dual reasoning, the functor
	\begin{align*}
	    G\colon \Ar(\Ccal) \crosslimits_\Ccal \Acal \crosslimits_\Ccal \Bcal_{/b_1}
	    &\too \Ccal \crosslimits_\Dcal \Dcal_{/\gbar(b_1)} \\
	    (\phi\colon c \to g(a), \psi \colon f(a) \to b_1) & \longmapsto (\gbar(\psi) \circ \fbar(\phi) \colon \fbar(c) \to \gbar(b_1))
	\end{align*}
	is a \emph{cartesian} fibration with weakly contractible fibers.
	
	In thus follows from \cref{lem:co/cart-contractible-fibers} below (setting $B = \Ccal$, $ p_0 = F $, and $ p_1 = G $) that the top functor in the following diagram is a weak equivalence:
	\begin{equation*}\begin{tikzcd}[column sep=2.5em]
	    \Bcal_{b_0/} \crosslimits_\Bcal \Acal \crosslimits_\Ccal \Ar(\Ccal) \crosslimits_\Ccal \Ar(\Ccal) \crosslimits_\Ccal \Acal \crosslimits_\Ccal \Bcal_{/b_1} \arrow[r, "F \times G"] \arrow[d]
	    & \Dcal_{\gbar(b_0)/} \crosslimits_\Dcal \Ccal \crosslimits_\Dcal \Dcal_{/\gbar(b_1)} \arrow[d] \\
	    \Bcal_{b_0/} \crosslimits_\Bcal \Acal \crosslimits_\Ccal \Ar(\Ccal) \crosslimits_\Ccal \Acal \crosslimits_\Ccal \Bcal_{/b_1} \arrow[r]
	    & \Map_{\Dcal \coproductlimits^\Ccal \Dcal}(\gbar(b_0), \gbar(b_1)) \period
	\end{tikzcd}\end{equation*}
	Here the bottom functor is defined analogously to $ F $ and $G$ by mapping to the pushout $\Dcal\coproductlimits^\Ccal \Dcal$ and composing all morphisms.
	The right vertical map is a weak equivalence by \cref{cor:map3types} applied to the pushout of the span $\Dcal \leftarrow \Ccal \to \Dcal$, where both arrows are fully faithful.
	The left functor has both adjoints because the composition functor
	\begin{equation*}
	    ([1] \xrightarrow{\delta^1} [2])^* \colon \Ar(\Ccal) \crosslimits_\Ccal \Ar(\Ccal) = \Fun([2], \Ccal)  \too \Fun([1],\Ccal) = \Ar(\Ccal)
	\end{equation*}
	has both adjoints and they are fiberwise over $\Ccal \times \Ccal$.
	In particular, the left functor is also a weak equivalence and so we conclude that the bottom functor is a weak equivalence as well.
	This is all that was left to show.
\end{proof}

\begin{lemma}\label{lem:co/cart-contractible-fibers}
    Let $p_0\colon \Ecal_0 \to \Bcal_0$ be a cocartesian fibration with weakly contractible fibers and let $p_1\colon \Ecal_1 \to \Bcal_1$ be a cartesian fibration with weakly contractible fibers.
    Then for any functors $\Bcal_0 \to \Bcal \leftarrow \Bcal_1$ the functor
    \begin{equation*}
        p_0 \times p_1\colon \Ecal_0 \cross_{\Bcal} \Ecal_1 \too \Bcal_0 \cross_{\Bcal} \Bcal_1
    \end{equation*}
    is a weak equivalence.
\end{lemma}

\begin{proof}
    We can factor the functor as
    \begin{equation*}
        \Ecal_0 \cross_{\Bcal} \Ecal_1 
        \too \Bcal_0 \cross_{\Bcal} \Ecal_1
        \too \Bcal_0 \cross_{\Bcal} \Bcal_1 \period
    \end{equation*}
    The first functor is a cartesian fibration with weakly contractible fibers and the second functor is a cocartesian fibration with weakly contractible fibers.
    (As they are basechanged from $p_0$ and $p_1$, respectively.)
    Therefore both are weak equivalences and hence so is their composite.
\end{proof}


\section{Computing pushouts via necklaces}\label{sec:computing_pushouts_via_necklaces}

Given any pushout square of \categories
\begin{equation}\label{eq:general_pushout_square}
    \begin{tikzcd}[column sep=2.5em]
        \Acal \arrow[d, "g"'] \arrow[r, "f"] \arrow[dr, phantom, very near end, "\ulcorner", xshift=0.25em, yshift=-0.25em] & \Bcal \arrow[d, "\gbar"] \\
        \Ccal \arrow[r, "\fbar"'] & \Dcal \comma
    \end{tikzcd}
\end{equation}  
we can take the pushout in $\Fun([1]\times [1], \Catinfty)$ of the square $(\gbar\colon \id{\Bcal} \to \id{\Dcal})$ with itself along the above square to obtain another pushout square whose horizontal maps are fold maps:
\begin{equation}\label{eqn:fold-pushout}
    \begin{tikzcd}[column sep=2.5em]
        \Bcal \coproductlimits^{\Acal} \Bcal \arrow[r, "\nabla"] \arrow[d, "\gbar \cup \gbar"'] \arrow[dr, phantom, very near end, "\ulcorner", xshift=0.25em, yshift=-0.25em] & \Bcal \arrow[d, "\gbar"] \\
        \Dcal \coproductlimits^{\Ccal} \Dcal \arrow[r, "\nabla"'] & \Dcal \period
    \end{tikzcd}
\end{equation}  
The goal of this section is to use the Segalification formula of \cite{arXiv:2301.0865} to prove that whenever $ f $ is fully faithful, this square has the following curious property.

\begin{proposition}\label{prop:fold-pushout}
    Consider a pushout square of \categories \eqref{eq:general_pushout_square} where $ f \colon \Acal \hookrightarrow \Bcal $ is fully faithful.
    For all $b_0,b_1 \in \Bcal \coproductlimits^{\Acal} \Bcal$ the square 
    \begin{equation*}
        \begin{tikzcd}[column sep=2.5em]
            \Map_{\Bcal \coproductlimits^{\Acal} \Bcal}(b_0, b_1) \arrow[r, "\nabla"] \arrow[d, "\gbar \cup \gbar"'] \arrow[dr, phantom, very near end, "\ulcorner", xshift=0.25em, yshift=-0.25em] & \Map_\Bcal(b_0, b_1) \arrow[d, "\gbar"] \\
            \Map_{\Dcal \coproductlimits^{\Ccal} \Dcal}(\gbar(b_0), \gbar(b_1)) \arrow[r, "\nabla"'] & \Map_{\Dcal}(\gbar(b_0), \gbar(b_1))
        \end{tikzcd}
    \end{equation*}  
    induced by \eqref{eqn:fold-pushout} on mapping anima is a pushout square.
\end{proposition}

Since the coproduct inclusions $\Bcal \hookrightarrow \Bcal \coproductlimits^\Acal \Bcal$ and $\Dcal \hookrightarrow \Dcal \coproductlimits^\Ccal \Dcal$ are both fully faithful, note that if $b_0 $ and $ b_1$ both lie in the left or both lie in the right copy of $\Bcal$ the two horizontal maps are equivalences. 
Therefore the new content of \cref{prop:fold-pushout} lies in the case where the $b_0$ and $ b_1$ are on opposite sides, and without loss of generality we will assume that $b_0$ is in the left and $b_1$ in the right copy.


\subsection{Recollection on necklaces}

To prove \cref{prop:fold-pushout} we use the Segalification formula in terms of necklaces, proved by Barkan and the third author in \cite{arXiv:2301.0865}.
Their formula was inspired by Dugger and Spivak's formula for the coherent nerve \cite{MR2764042}.

\begin{recollection}
	There is an adjunction
	\begin{equation*}  
		\adjto{\ac}{\sSpc}{\Catinfty}{\Nerve}
	\end{equation*}
	where the \textit{associated category} functor $\ac$ is left Kan extended from the full inclusion $ \incto{\DDelta}{\Catinfty} $.
	The right adjoint $\Nerve$ is the Rezk nerve, which assigns to \acategory $\Ccal$ the simplicial anima
	\begin{equation*}
		\Nerve_n(\Ccal) = \Map_{\Catinfty}([n], \Ccal) = \Fun([n],\Ccal)^{\equivalent} \period
	\end{equation*}
\end{recollection}

Since $ \ac $ is a Bousfield localization, we can compute a pushout of \categories by computing the pushout of their nerves and then applying $\ac$ to the result.
In order to compute the fourth mapping anima we recall from \cite{arXiv:2301.0865} a formula for the mapping anima in the associated category $\ac(X)$ of a simplicial anima $ X $.
This uses the notion of \textit{necklaces}, which we now recall.

\begin{figure}[h]
    \centering
    \includegraphics[width = .5\linewidth]{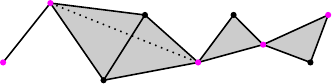}
    \caption{The necklace $N = \Delta^1 \vee \Delta^3 \vee \Delta^2 \vee \Delta^2$, with its joints marked in pink.
    The alternative notation for this necklace is $([8], \{0,1,4,6,8\})$.}
    \label{fig:necklace}
\end{figure}

\begin{recollection}[(Necklaces)]
	Let $\sSet_{**} \colonequals \sSet_{*\sqcup */}$ be the $ 1 $-category of bipointed simplicial sets.
	Then $ \sSet_{**} $ has a monoidal structure defined by
	\begin{equation*}
	    (A, a_0, a_1) \vee (B, b_0, b_1) \coloneq (A \coproduct^{a_1=b_0} B, a_0, b_1)
	\end{equation*}
	and we define the $ 1 $-category of necklaces to be the smallest full subcategory $\Nec \subset \sSet_{**}$ that contains all simplices $([n], 0, n)$ and is closed  under the monoidal product.
	For every necklace $N$ there is a unique, strictly monotone sequence of integers $0 = a_0 < a_1 < \dots < a_k = n$ such that
	\begin{equation*}
	    N \isomorphic \Delta^{a_1-a_0} \vee \dots \vee \Delta^{a_k-a_{k-1}} \period
	\end{equation*}
	The simplices $\Delta^{a_i-a_{i-1}} \subset N$ are called \emph{maximal simplices} and the subset $\{a_0, \dots, a_n\} \subset N$ is called the set of \emph{joints} of $N$, denoted $ \Joints(N) $.
	The set of joints $\Joints(N)$ may equivalently be described as the subset of those $0$-simplices of $N$ that do not appear as the middle vertex of a non-degenerate $2$-simplex in $N$.

	If $[n] \in \DDelta$ and $A \subset [n]$ is a subset containing $0$ and $n$, then we let $([n], A)$ denote the necklace defined as the unique subnecklace of $\Delta^n$ with joints $A$.
	A morphism $d\colon [n] \to [m]$ in $\DDelta$ to induces a map of necklaces $([n], A) \to ([m], B)$ if and only if it satisfies $d(A) \supset B$. 
	In particular, $d$ must be an active morphism.
	In other words, the functor
	\begin{equation*}
	    \Nec \longrightarrow \Deltaact \comma \qquad
	    N = ([n], A) \longmapsto [n]
	\end{equation*}
	is a cocartesian fibration and its straightening sends $[n]$ to the (opposite of the) poset of subsets $A \subset [n]$ that contain $0$ and $n$.
\end{recollection}

\begin{nul}
	For a simplicial anima $ X $ and two points $x_0, x_1 \in X_0$ we would like to compute the mapping anima $\Map_{\ac(X)}(x_0, x_1)$.
	The triple $(X, x_0, x_1)$ defines an object in the \category 
	\begin{equation*}
		\sSpc_{**} \coloneq \sSpc_{*\sqcup */}
	\end{equation*}
	of bipointed simplicial anima.
	The main theorem of \cite{arXiv:2301.0865} now says that the mapping anima can be computed as
	\begin{equation*}
	    \Map_{\ac(X)}(x_0, x_1) \simeq \colim_{N \in \Necop} \Map_{**}(N, (X, x_0, x_1)) \period
	\end{equation*}
	The unstraightening of the presheaf $\Map_{**}(-, (X, x_0, x_1))$ is given by
	\begin{equation*}
	    \Nec_{x_0, x_1}(X) \coloneq \Nec \crosslimits_{\sSpc_{**}} (\sSpc_{**})_{/(X, x_0, x_1)} \period
	\end{equation*}
	Therefore, the mapping anima in $\ac(X)$ can also be described as
	\begin{equation*}
	    \Map_{\ac(X)}(x_0, x_1) \simeq |\Nec_{x_0, x_1}(X)| \period
	\end{equation*}
	We mainly apply this in the case where $ X = \Nerve(\Bcal) \coproduct^{\Nerve(\Acal)} \Nerve(\Ccal) $ for some span $ \Bcal \hookleftarrow \Acal \rightarrow \Ccal $.
	However, while studying this situation, we also need to consider a few variants of this pushout.
\end{nul}


\subsection{A pushout in terms of constrained necklaces}

In order to simplify the colimit over the necklace category we use the subpresheaf
\begin{equation*}
	\Map^{(A)}(-,X) \subset \Map(-, X)
\end{equation*}
on $\Nec$ of \textit{constrained necklaces}, which is defined by requiring that the necklace passes through $A$ (with one of its joints).

\begin{definition}
    For a bipointed simplicial anima $(X, x_0, x_1) \in \sSpc_{**}$, a union of components $A \subset X_0$, and $N \in \Nec$ we define the anima of \defn{$ A $-constrained necklace maps}
    \begin{equation*}  
        \Map^{(A)}_{**}(N, X) \subset 
        \Map_{**}(N, X)
    \end{equation*}
    as the subanima of those maps $f\colon N \to X$ such that there is a joint $j \in \Joints(N)$ that is mapped to $f(j) \in A$.
    Note that this defines a subpresheaf because if $d\colon M \to N$ is a map of necklaces, then $d(\Joints(M)) \supset \Joints(N)$ and hence we can find $j' \in M$ with $d(j') = j$ and $f(d(j')) = f(j) \in A$.
\end{definition}

\begin{notation}
	Throughout the rest of this section, we fix a span of \categories $ \Bcal \hookleftarrow \Acal \rightarrow \Ccal $ where the left-hand functor is fully faithful.
\end{notation}

\begin{notation}
	We also introduce some notation for the simplicial anima $P = \Nerve(\Bcal) \coproduct^{\Nerve(\Acal)} \Nerve(\Ccal)$ that is the pushout of the nerves.
	Since $\Nerve(\Acal) \hookrightarrow \Nerve(\Bcal)$ is levelwise a monomorphism, $\Nerve_n \Bcal$ decomposes as a coproduct of $\Nerve_n \Acal$ and its complement $Q_n\subset \Nerve_n \Bcal$.
	We have a similar decomposition of $P_n$ with the same complement so that
	\begin{equation*}
		\Nerve_n \Bcal = (\Nerve_n \Acal) \sqcup Q_n
		\qquad \text{and} \qquad
		P_n = (\Nerve_n \Ccal) \sqcup Q_n \period
	\end{equation*}
	Note, however, that $Q_n$ is \emph{not} a simplicial anima.
	We further let $Q_n' \subset Q_n$ denote the subanima of those $n$-simplices where neither the $0$th nor the $n$-th vertex is in $\Nerve_0 \Acal$. 
	(Here $Q_0' = Q_0$.)
\end{notation}

We can now write the fourth mapping anima as a pushout of other anima, which we will have to study in more detail below.

\begin{lemma}\label{lem:Map(A)-pushout}
    Let $\Dcal$ be the pushout of $\Bcal \hookleftarrow \Acal \rightarrow \Ccal$ 
    and let $P = \Nerve(\Bcal) \coproduct^{\Nerve(\Acal)} \Nerve(\Ccal)$ be the pushout of nerves taken in $\sSpc$.
    For all $b_0, b_1 \in \Bcal$ there is a pushout square
    \begin{equation*}
    	\begin{tikzcd}
	        \colim_{N \in \Necop} \Map_{**}^{(\Acal)}(N, \Nerve(\Bcal)) \arrow[r] \arrow[d] \arrow[dr, phantom, very near end, "\ulcorner", xshift=0.25em] &
	        \Map_\Bcal(b_0, b_1) \arrow[d] \\
	        \colim_{N \in \Necop} \Map_{**}^{(\Ccal)}(N, P) \arrow[r] &
	        \Map_\Dcal(g(b_0), g(b_1)) \period
	    \end{tikzcd}
    \end{equation*}
\end{lemma}

\begin{proof}
    Consider the following square in $\Fun(\Necop, \Spc)$:
    \begin{equation}\label{eqn:Nec-psh-pushout}
	    \begin{tikzcd}
	        \Map_{**}^{(\Acal)}(-, \Nerve(\Bcal)) \arrow[r, hooked] \arrow[d] &
	        \Map_{**}(-, \Nerve(\Bcal)) \arrow[d] \\
	        \Map_{**}^{(\Ccal)}(-, P) \arrow[r, hooked] &
	        \Map_{**}(-, P) \period
	    \end{tikzcd}
    \end{equation}
    For a given necklace $N = \Delta^{n_1} \vee \dots \vee \Delta^{n_k}$ we have a decomposition
    \begin{align*}
        \Map_{**}(N, \Nerve(\Bcal)) &= \Nerve_{n_1}\Bcal \crosslimits_{\Nerve_0\Bcal} \cdots \crosslimits_{\Nerve_0\Bcal} \Nerve_{n_k}\Bcal \\
        &= \Map_{**}^{(\Acal)}(N, \Nerve(\Bcal)) \sqcup \left(Q_{n_1}'\times_{Q_0} \dots \times_{Q_0} Q_{n_k}'\right)
    \end{align*}
    because the complement of $\Map_{**}^{(\Acal)}(N, \Nerve(\Bcal)) \subset \Map_{**}(N, \Nerve(\Bcal))$ consists of all those necklaces in $\Bcal$ for which \emph{every} joint is not in $\Acal$ and therefore each maximal simplex is in $Q_{n_i}' \subset \Nerve_{n_i}\Bcal$ as both its $0$th and last vertex are in $Q_0$.
    By the same argument we also have a decomposition
    \begin{equation*}
        \Map_{**}(N, P) 
        = \Map_{**}^{(\Ccal)}(N, P) \sqcup \left(Q_{n_1}'\times_{Q_0} \dots \times_{Q_0} Q_{n_k}'\right) \period
    \end{equation*}
    Therefore, in the square (\ref{eqn:Nec-psh-pushout}) the two horizontal maps are monomorphisms and their complements are identified by the vertical maps, so the square is a pushout square of presheaves of anima on $\Nec$.
    Taking the colimit over $\Necop$ we obtain a pushout square
    \begin{equation*}\begin{tikzcd}
        \colim_{N \in \Necop} \Map_{**}^{(\Acal)}(N, \Nerve(\Bcal)) \arrow[r] \arrow[d] \arrow[dr, very near end, phantom, "\ulcorner"] &
        \colim_{N \in \Necop} \Map_{**}(N, \Nerve(\Bcal)) \arrow[d] \\
        \colim_{N \in \Necop} \Map_{**}^{(\Ccal)}(N, P) \arrow[r] &
        \colim_{N \in \Necop} \Map_{**}(N, P) \period
    \end{tikzcd}\end{equation*}
    Since $\ac(\Nerve(\Bcal)) = \Bcal$ and $\ac(P) = \Dcal$ the two right-hand terms are, by the main theorem of \cite{arXiv:2301.0865}, equivalent to $\Map_\Bcal(b_0, b_1)$ and $\Map_\Dcal(g(b_0), g(b_1))$, respectively.
\end{proof}


\subsection{Segal conditions away from subsets}

The remainder of the proof of \cref{prop:fold-pushout} is about identifying the left two terms in the pushout square in \cref{lem:Map(A)-pushout}.
For this we need the following criterion:

\begin{definition}
    For a simplicial anima $ X $ we say that a simplicial subanima $A \subset X$ is \defn{full} if each $A_n \hookrightarrow X_n$ is a monomorphism and an $n$-simplex is in $A_n$ if and only if all its vertices are in $A_0$.
    We say $ X $ is \emph{Segal away from $A$} if it satisfies the following:
    for any map of necklaces $d\colon N \to N'$ that is a bijection on vertices, the restriction map
    \begin{equation*}
        d^*\colon \Map(N', X) \longrightarrow \Map(N, X)
    \end{equation*}
    has contractible fibers at all those points $f\colon N = ([n], J) \to X$ with $d(f^{-1}(A_0) \cap J) \subset J'$.
    In other words, if $f\colon ([n], J) \to X$ is a necklace in $ X $ and $j \in J \setminus \{\bot,\top\}$ is an inner joint with $f(j) \not\in A_0$, then $ f $ extends uniquely to $([n], J\setminus \{j\})$. 
\end{definition}

\begin{observation}\label{obs:Segal-away-Delta-n}
    A simplicial anima $ X $ is Segal away from a full subanima $A$ if and only if for every necklace $N = ([n], J)$ the map
    \begin{equation*}
        \Map_{\sSpc}(\Delta^n, X) \longrightarrow \Map_{\sSpc}(N, X)
    \end{equation*}
    has contractible fibers over all those $f\colon N \to X$ that send no inner joints to $A$, i.e., 
    \begin{equation*}
    	f(J\setminus \{\bot,\top\}) \subset X_0 \setminus A_0 \comma
    \end{equation*}
    or equivalently $f^{-1}(A_0) \cap J \subset \{\bot,\top\}$.
    Indeed, this implies the general condition because if $d\colon N \to N'$ is a more general map we can write $N' = \Delta^{n_1} \vee \dots \vee \Delta^{n_k}$ and this induces a decomposition $N = N_1 \vee \dots \vee N_k$ by taking preimages.
    Therefore, the map $d$ can be written as $d=d_1 \vee \dots \vee d_k$ such that each $d_i$ has a full simplex as its target.
    We know that each of the $d_i^*$ have contractible fibers at the relevant points and writing $d^*$ as an iterated pullback of the $d_i^*$ in $\Ar(\Spc)$ we get that $d^*$ also has contractible fibers at the points in question.
\end{observation}

We have plenty of examples of this property because it has the following stability under pushouts.

\begin{lemma}\label{lem:Segal-away-pushout}
    Let $A \subset X$ be a full simplicial subanima and $A \to C$ any map.
    If $ X $ is Segal away from $A$, then $C \coproduct^A X$ is Segal away from $C$.
\end{lemma}

\begin{proof}
    We can make a level-wise decomposition $X_n = A_n \sqcup Q_n$.
    Write $N=\Delta^{n_1}\vee \dots \vee \Delta^{n_k}$.
    Then the anima of necklace maps $N \to X$ that send no inner joints to $A$ is 
    \begin{equation*}
        Q_{n_1} \crosslimits_{X_0} Q_{n_2}' \crosslimits_{X_0} \dots \crosslimits_{X_0} Q_{n_{k-1}}' \crosslimits_{X_0} Q_{n_k} 
        \subset X_{n_1} \crosslimits_{X_0} X_{n_2} \crosslimits_{X_0} \dots \crosslimits_{X_0} X_{n_k} = \Map_\sSpc(N, X) \period
    \end{equation*}
    After taking the pushout we still have a decomposition $(C \coproduct^A X)_n = C_n \sqcup Q_n$ and so the anima of necklaces in $P=C \coproduct^A X$ that send no inner vertices to $C$ is 
    \begin{equation*}
        Q_{n_1} \crosslimits_{(C \coproduct^A X)_0} Q_{n_2}' \crosslimits_{(C \coproduct^A X)_0} \dots \crosslimits_{(C \coproduct^A X)_0} Q_{n_{k-1}}' \crosslimits_{(C \coproduct^A X)_0} Q_{n_k} 
        \subset \Map_\sSpc(N, C \coproduct^A X) \period
    \end{equation*}
    In both cases the fiber products over $X_0$ or $(C \coproduct^A X)_0$ really only map to the subanima $Q_0$, so that the map $X \to C \coproduct^A X$ identifies these two subanima of $\Map_\sSpc(N, X)$ and $\Map_\sSpc(N, C \coproduct^A X)$.
    As it also identifies the subanima of
    \begin{equation*}
    	Q_n \subset \Map(\Delta^n, X) \andeq Q_n \subset \Map(\Delta^n, C \coproduct^A X) \comma
    \end{equation*}
    the claim follows from \cref{obs:Segal-away-Delta-n}.
\end{proof}

\begin{example}\label{ex:Segal-away}
	\hfill
    \begin{enumerate}
        \item If $ X $ is Segal, then $ X $ is Segal away from any full simplicial subanima $A \subset X $, since the map $N \to \Delta^n$ is local with respect to Segal anima.
        In fact, $ X $ is Segal if and only if it is Segal away from $\emptyset$.

        \item If $P= \Nerve(\Bcal) \coproduct^{\Nerve(\Acal)} \Nerve(\Ccal)$ as before, then $P$ is Segal away from $\Nerve(\Ccal)$.
        This follows from \cref{lem:Segal-away-pushout} because, by the previous example, $\Nerve \Bcal$ is Segal away from $\Nerve \Acal$.

        \item\label{it:double-Segal-away} If $A \subset X$ is a full simplicial subanima such $ X $ is Segal away from $A$, then $X \coproduct^A X$ is Segal away from $A$.
        Indeed, if $f\colon N \to X \coproduct^A X$ is a necklace sending no inner joints to $A$, then $ f $ must land entirely in one of the copies of $ X $.
        In other words, $ f $ does not send any inner joints to the other copy $X \subset X \coproduct^A X$, but by \cref{lem:Segal-away-pushout} $X \coproduct^A X$ is Segal away from $ X $, so we are done.
    \end{enumerate}
\end{example}


\subsection{Constrained necklaces and fold maps}

Under the assumption of Segalness away from $A \subset X$ we have a concrete interpretation of the colimit of the $\Nec$-presheaf $\Map_{**}^{(A)}(-, X)$.
This involves the fold/codiagonal map
\begin{equation*}
    \nabla_X \colon X \coproduct^A X \longrightarrow X \period
\end{equation*}

\begin{lemma}\label{lem:Map(A)=fold}
	Let $ X $ be a simplicial anima and let $A \subset X$ be a full simplicial subanima such that $ X $ is Segal away from $A$.
    Let $x_L, x_R \in (X \coproduct^A X)_0$ be vertices such that $x_L$ is in the left copy of $ X $ and $x_R$ is in the right copy of $ X $.
    Then there are equivalences
    \begin{equation*}
    	\begin{tikzcd}
    		\colim_{N \in \Necop} \Map_{**}^{(A)}(N, (X \coproduct^A X,x_L,x_R)) \arrow[r, "\sim"{yshift=-0.25em}, "\nabla_X"'] \arrow[d, "\wr"{xshift=-0.15em}] & \colim_{N \in \Necop} \Map_{**}^{(A)}(N, (X,x_L,x_R)) \\ 
    		\Map_{\ac(X) \coproduct^{\ac(A)} \ac(X)}(x_L, x_R) & \phantom{\colim_{N \in \Necop} \Map_{**}^{(A)}(N, (X,x_L,x_R))} \period
    	\end{tikzcd}
    \end{equation*}
\end{lemma}

To prove this we need some auxilliary versions of the category of necklaces in $ X $.

\begin{definition}
    For a simplicial anima $ X $, a full simplicial subanima $A \subset X$, and $x, y \in X_0 \setminus A_0$ we define the full subcategories
    \begin{equation*}
        \Nec_{x,y}^{A}(X) \subset
        \Nec_{x,y}^{[A]}(X) \subset
        \Nec_{x,y}^{(A)}(X) \subset
        \Nec_{x,y}(X)
    \end{equation*}
    where, as illustrated in \cref{fig:restricted-necklaces}, an object $(f\colon N \to X) \in \Nec_{x,y}(X)$ is:
    \begin{enumerate}
        \item In $\Nec_{x,y}^{(A)}(X)$ if it maps at least one joint to $A$, i.e., $\Joints(N) \cap f^{-1}(A) \neq \emptyset$.

        \item In $\Nec_{x,y}^{[A]}(X)$ if $f \in \Nec_{x,y}^{(A)}(X)$ and $ f $ maps all inner joints to $A$, i.e., $f(\Joints(N) \setminus \{\bot, \top\}) \subset A$.

        \item In $\Nec_{x,y}^A(X)$ if $f \in \Nec_{x,y}^{(A)}(X)$ and $ f $ maps all inner vertices to $A$, i.e., $f(N_0\setminus \{\bot,\top\}) \subset A$.
    \end{enumerate}
    Note that $\Nec_{x,y}^{(A)}(X) \to \Nec$ is the right fibration that straightens to the presheaf $\Map_{\Nec}^{(A)}(-, (X,x,y))$.
\end{definition}

\begin{figure}[h]
    \centering
    \includegraphics[width = .7\linewidth]{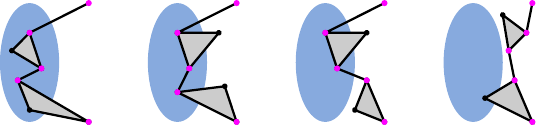}
    \caption{Four necklaces that are in $\Nec_{x,y}^A(X)$, $\Nec_{x,y}^{[A]}(X)$, $\Nec_{x,y}^{(A)}(X)$ and $\Nec_{x,y}(X)$, respectively. (They are chosen such that each of them is not in any of the smaller subcategories.)}
    \label{fig:restricted-necklaces}
\end{figure}

\begin{lemma}\label{lem:MapA=Map(A)}
    For any full simplicial subanima $A \subset X$ the full inclusion
    \begin{equation*}
        \Nec_{x,y}^{A}(X) \hookrightarrow \Nec_{x,y}^{[A]}(X)
    \end{equation*}
    admits a right adjoint.
    If we additionally assume that $ X $ is Segal away from $A$, then the full inclusion
    \begin{equation*}
        \Nec_{x,y}^{[A]}(X) \hookrightarrow \Nec_{x,y}^{(A)}(X)
    \end{equation*}
    admits a left adjoint.
    In particular, under these assumptions, both inclusions are weak equivalences.
\end{lemma}

\begin{figure}[h]
    \centering
    \includegraphics[width = .8\linewidth]{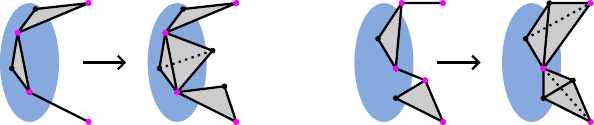}
    \caption{The left arrow depicts the counit of the colocalization $\Nec_{x,y}^A(X) \rightleftarrows \Nec_{x,y}^{[A]}(X)$.
    The right arrow depicts the unit of the localization $ \Nec_{x,y}^{(A)}(X) \rightleftarrows \Nec_{x,y}^{[A]}(X)$.}
    \label{fig:adjoints}
\end{figure}

\begin{proof}
    The right adjoint to the first inclusion is defined by sending $(f\colon N = ([n], J) \to X)$ to $(\restrict{f}{N'}\colon N' \to X)$ where $N' \subset N$ is the full subnecklace spanned by the vertices $\{\bot, \top\} \cup f^{-1}(A)$. 
    This is well-defined because we know that $ f $ already sends all inner joints to $A$, so all of the joints in $N$ are still contained in $N'$.
    \cref{fig:adjoints} depicts both the counit of this adjunction and the unit of the adjunction we are about to construct.

    Now assume that $ X $ is Segal away from $A$.
    For any $(f\colon N = ([n], J) \to X) \in \Nec_{x,y}^{(A)}(X)$ we need to construct a localization onto $\Nec_{x,y}^{[A]}(X)$.
    Let
    \begin{equation*}
    	J_0 \colonequals \{\bot, \top\} \cup (f^{-1}(A) \cap J) \subset J
    \end{equation*}
    be the set of the two extremal vertices and those inner joints that are mapped to $A$.
    We can define a new necklace by $N' \coloneq ([n], J_0)$ and the identity on $[n]$ induces an inclusion $\iota \colon N \hookrightarrow N'$.
    Because $ X $ is Segal away from $A$ the map $f\colon N \to X$ extends uniquely a map $f'\colon N' \to X$.
    %
    We thus have a map 
    \begin{equation*}
        \iota\colon (f\colon N \to X) \longrightarrow (f'\colon N' \to X)
    \end{equation*}
    in $\Nec_{x,y}^{(A)}(X)$
    and $(f'\colon N' \to X) \in \Nec_{x,y}^{[A]}(X)$ because we have forgotten all inner joints of $N$ that did not map to $A$.

    It remains to check that $f'$ is indeed a localization of $ f $. 
    We need to show that for any object in the subcategory $(g\colon L \to X) \in \Nec_{x,y}^{[A]}(X)$, precomposition by $\iota$ induces an equivalence of mapping anima into $g$.
    We can write these mapping anima as fibers, so that the map in question is the left map in the diagram
	\begin{equation*}
		\begin{tikzcd}[sep=3em]
			\Map_{\Nec_{x,y}(X)}(f', g) \arrow[r] \arrow[d, "-\circ \iota"'] & \Map_{\Nec}(N', L) \arrow[r, "g \of -"] \arrow[d, "-\circ \iota"{description}] & \Map_{\sSpc}(N', X) \arrow[d, "-\circ \iota"] \\
			\Map_{\Nec_{x,y}(X)}(f, g) \arrow[r] & \Map_{\Nec}(N, L) \arrow[r, "g \of -"'] & \Map_{\sSpc}(N, X)
		\end{tikzcd}
	\end{equation*}
	whose horiztonal lines are fiber sequences.
    Consider a morphism $f \to g$ given by $d\colon N \to L$ and a homotopy $g \circ d \simeq f$.
    The middle map is injective because $N \hookrightarrow N'$ is an epimorphism in $\Nec$.
    Its fiber over $d$ is nonempty because if $j\in N$ is a joint with $f(j) \simeq g(d(j)) \not\in A$, then $d(j) \in L$ cannot be an inner joint of $L$ (as the inner joints of $L$ are sent to $A$) and thus $d$ also yields a well-defined map $N' \to L$ after deleting those joints.
    The right map has contractible fibers on the component of $ f $ because we showed that $ f $ extends uniquely (as $ X $ is Segal away from $A$).
    This shows that the left map has a contractible fiber over $d\colon f\to g$; since $d$ was arbitary, we conclude that the left map is an equivalence.
\end{proof}

With this tool in hand we are ready to prove \cref{lem:Map(A)=fold}.

\begin{proof}[Proof of \cref{lem:Map(A)=fold}]
    First, we note that for a necklace in $X \coproduct^A X$ to start at $x_L$ and end at $x_R$ it must have an at least one joint in $A_0$.
    Either $x_L \in A_0$, or $x_R \in A_0$, or they are in distinct copies of $X_0 \setminus A_0$ and to get from one to the other we need a joint in $A_0$ as every maximal simplex of the necklace must entirely map to one copy of $ X $.
    In formulas we have
    \begin{equation*}
        \Map_{**}^{(A)}(-, (X \coproduct^A X,x_L,x_R)) = \Map_{**}(-, (X \coproduct^A X,x_L,x_R)) \period
    \end{equation*}
    After taking colimits, by the main theorem of \cite{arXiv:2301.0865}, the right term computes the mapping anima
    \begin{equation*}
        \colim_{N \in \Necop} \Map_{**}(N, (X \coproduct^A X, x_L, x_R)) 
        \simeq \Map_{\ac(X \coproduct^A X)}(x_L, x_R) \period
    \end{equation*}
    Since $ \ac(X \coproduct^A X) \simeq \ac(X) \coproduct^{\ac(A)} \ac(X) $, this proves the horizontal equivalence in the statement of \cref{lem:Map(A)=fold}.

    To show the vertical equivalence, we use that we can compute the colimits by unstraightening and inverting all morphisms.
    In these terms, we need to show that the right vertical map in the square
    \begin{equation*}
	    \begin{tikzcd}
	    	\Nec_{x_L,x_R}^{A}(X\coproduct^A X) \arrow[r, hooked] \arrow[d, "\nabla_X"'] & \Nec_{x_L,x_R}^{(A)}(X\coproduct^A X) \arrow[d, "\nabla_X"] \\
	    	\Nec_{x_L,x_R}^{A}(X) \arrow[r, hooked] & \Nec_{x_L,x_R}^{(A)}(X)
	    \end{tikzcd}
    \end{equation*}
    is a weak equivalence.
    Since $ X $ is Segal away from $A$ and $X \coproduct^A X$ is Segal away from $A$ by \cref{ex:Segal-away}, \cref{lem:MapA=Map(A)} shows that the horizontal maps are weak equivalences.
    It remains to show that the left map is a weak equivalence, and in fact we show that it is an equivalence of \categories.

    The left functor is obtained by restricting the map of right fibrations
    \begin{equation*}
    	\Nec_{x_L, x_R}(X \coproduct^A X) \to \Nec_{x_L, x_R}(X)
    \end{equation*}
    over $\Nec$ to certain full subcategories.
    Therefore, it suffices to check that it induces an equivalence on the fibers over each $N \in \Nec$. 
    The induced map on fibers is the map
    \begin{equation*}
        \nabla_X \circ -\colon \Map_{**}^{A}(N, (X \coproduct^A X, x_L, x_R)) \longrightarrow \Map_{**}^{A}(N, (X, x_L, x_R)) \period
    \end{equation*}
    If $N = \Delta^{n_1} \vee \dots \vee \Delta^{n_k}$ then because all inner vertices have to map to $\Acal$ both sides evaluate to (here we use that $A \subset X$ and $A \subset X \coproduct^A X$ are full)
    \begin{equation*}  
        \{x_L\} \crosslimits_{X_0} X_{n_1}
        \crosslimits_{X_{n_1-1}} A_{n_1-1}
        \crosslimits_{X_0} A_{n_2}
        \crosslimits_{X_0} 
        \cdots
        \crosslimits_{X_0} A_{n_{k-1}}
        \crosslimits_{X_0} A_{n_k-1}
        \crosslimits_{X_{n_k-1}} X_{n_k}
        \crosslimits_{X_0} \{x_R\}
    \end{equation*}
    and the map $\nabla_X \circ -$ described above is an equivalence, completing the proof.
\end{proof}


\subsection{Proof of \texorpdfstring{\Cref{prop:fold-pushout}}{Proposition \ref*{prop:fold-pushout}}}

By \cref{lem:Map(A)-pushout} the lower square in the diagram 
\begin{equation*}
	\begin{tikzcd}[sep=2.5em]
		\colim_{N \in \Necop} \Map_{**}(N, (\Nerve\Bcal \coproductlimits^{\Nerve\Acal} \Nerve\Bcal, b_0^L, b_1^R)) \arrow[d, "\nabla", "\wr"'{xshift=0.15em}] \arrow[r] & \colim_{N \in \Necop} \Map_{**}(N, (P \coproductlimits^{\Nerve\Ccal} P, g(b_0^L), g(b_1^R)))  \arrow[d, "\nabla", "\wr"'{xshift=0.15em}] \\
		\colim_{N \in \Necop} \Map_{**}^{(\Acal)}(N, (\Nerve\Bcal, b_0, b_1)) \arrow[r] \arrow[d] \arrow[dr, "\ulcorner", very near end, phantom] & \colim_{N \in \Necop} \Map_{**}^{(\Ccal)}(N, (P, g(b_0), g(b_1)) \arrow[d] \\ 
		\Map_\Bcal(b_0, b_1) \arrow[r] &  \Map_\Dcal(g(b_0), g(b_1))
	\end{tikzcd}
\end{equation*}
is a pushout square and by \cref{lem:Map(A)=fold} the vertical maps in the top square are equivalences.
The top-most terms compute mapping anima in the associated categories.
Because $\ac(-)$ commutes with colimits we compute
\begin{equation*}
    \ac(P \coproductlimits^{\Nerve\Ccal} P) 
    \simeq \ac(P) \coproductlimits^{\ac(\Nerve\Ccal)} \ac(P) 
    \simeq \Dcal \coproductlimits^{\Ccal} \Dcal \period
\end{equation*}
Therefore we can rewrite the outside square in the above diagram as
\begin{equation*}
	\begin{tikzcd}[sep=2.5em]
		\Map_{\Bcal \coproductlimits^{\Acal} \Bcal}(b_0^L, b_1^R) \arrow[d, "\nabla"'] \arrow[r] \arrow[dr, "\ulcorner", very near end, phantom] & \Map_{\Dcal \coproductlimits^{\Ccal} \Dcal}(g(b_0^L), g(b_1^R))) \arrow[d, "\nabla"] \\
		\Map_\Bcal(b_0, b_1) \arrow[r] & \Map_\Dcal(g(b_0), g(b_1))
	\end{tikzcd}
\end{equation*}
which is exactly the pushout square claimed in \cref{prop:fold-pushout}. \hfill $ \qed $


\section{Examples and applications}\label{sec:examples_and_applications}

The goal of this section is to give sample applications and examples of the main results in this paper. 
\Cref{section:sieves} has examples involving pushouts along sieve inclusions, \cref{subsec:Dwyer_functors} concerns Dwyer functors, and \cref{subsec:pushout_products} concerns pushout products.
\Cref{subsec:Reedy_categories} is lengthier and concerns examples involving Reedy \categories.


\subsection{Sieves}\label{section:sieves}

In this example, we start by giving a simpler proof of the result of \Cref{section:fourth} when $ f \colon \incto{\Acal}{\Bcal} $ is a sieve inclusion. 
Recall that a sieve is a full subcategory $\Acal \subset \Bcal$ such that for every $f\colon b \to a$ in $\Bcal$ with $a \in \Acal$ we also have $b \in \Bcal$.

\begin{proposition}
    Let
    \begin{equation*}
        \begin{tikzcd}[column sep=2.5em]
            \Acal \arrow[r, "g"] \arrow[d, "f"', hooked] \arrow[dr, phantom, very near end, "\ulcorner", xshift=0.25em, yshift=-0.25em] & \Ccal \arrow[d, "\fbar", hooked] \\
            \Bcal \arrow[r, "\gbar"'] & \Dcal 
        \end{tikzcd}
    \end{equation*}  
    be a pushout square of \categories where $ f $ is fully faithful.
    Let $b,b'\in\Bcal$ and assume that $b$ does not map to any $a\in \Acal$.
    Then the map $\Map_{\Bcal}(b,b')\to \Map_{\Dcal}(\gbar(b),\gbar(b'))$ is an equivalence.  
\end{proposition}

\begin{proof}
     Consider the functor $\widetilde{\Map}(b,-) \colon \Dcal\to \Spc$ defined on $\Bcal$ by $\Map_{\Bcal}(b,-)$, on $\Ccal$ by the constant functor $\emptyset$, and their unique equivalence on $\Acal$ (they \emph{are} equivalent on $\Acal$ by the assumption that $b$ has no map to any $a\in \Acal$). 

    Let $F \colon \Dcal\to\Spc$ be any functor. 
    Using the Yoneda lemma and the fact that a constant diagram with value the empty set is initial a presheaf \category, we compute
    \begin{align*}
    	\Map(\widetilde{\Map}(b,-),F) &\simeq \Map_{\Fun(\Bcal,\Spc)}(\Map_{\Bcal}(b,-),F\circ \gbar) \crosslimits_{\Map_{\Fun(\Acal,\Spc)}(\emptyset, F\circ \gbar f)}\Map_{\Fun(\Ccal,\Spc)}(\emptyset, F\circ \fbar) \\ 
    	&\simeq \Map_{\Fun(\Bcal,\Spc)}(\Map_{\Bcal}(b,-),F\circ g)\simeq F(g(b)) \period
    \end{align*} 
    By the Yoneda lemma again, it follows that $\widetilde{\Map}(b,-)\simeq \Map_{\Dcal}(g(b),-)$, and one easily checks that the desired map is the canonical one.
\end{proof}

\begin{corollary}\label{cor:sieve}
	Let 
	\begin{equation*}
		\begin{tikzcd}[column sep=2.5em]
		    \Acal \arrow[r, "g"] \arrow[d, "f"', hooked] \arrow[dr, phantom, very near end, "\ulcorner", xshift=0.25em, yshift=-0.25em] & \Ccal \arrow[d, "\fbar"] \\
		    \Bcal \arrow[r, "\gbar"'] & \Dcal \comma
		\end{tikzcd}
	\end{equation*}  
    be a pushout square of \categories where $ f $ is a sieve inclusion.
    Then for all $b,b'\in\Bcal$, exactly one of the two following situations happens:
    \begin{enumerate}
        \item We have $b\in \Acal$.
        In this case,
        \begin{equation*}
        	\Map_{\Dcal}(\gbar(b),\gbar(b')) \simeq \Map_{\Dcal}(\fbar(g(b)),\gbar(b'))\simeq |\Ccal_{g(b)/}\crosslimits_{\Ccal}\Acal\crosslimits_{\Bcal}\Bcal_{/b'}| \period
        \end{equation*}  

        \item We have $b\notin\Acal$.
        In which case $\Map_{\Dcal}(\gbar(b),\gbar(b'))\simeq \Map_{\Bcal}(b,b')$. 
    \end{enumerate}
\end{corollary}

We also obtain:

\begin{corollary}\label{cor:pushouts_of_sieves_is_union}
	Let $ \Ccal $ be \acategory and let $ \Ccal_0, \Ccal_1 \subset \Ccal $ be sieves.
	Then the natural square
	\begin{equation*}
        \begin{tikzcd}[column sep=2.5em]
            \Ccal_0 \intersect \Ccal_1 \arrow[r, hooked] \arrow[d, hooked] & \Ccal_1 \arrow[d, hooked] \\
             \Ccal_0 \arrow[r, hooked] &  \Ccal_0 \union \Ccal_1 
        \end{tikzcd}
    \end{equation*}  
    is a pushout square of \categories.
\end{corollary}

\begin{proof}
    By \cref{cor:ff-and-ff} it suffices to inspect the ``cross terms'' $\Map_{\Ccal} (c_0,c_1)$ and $\Map_{\Ccal}(c_1,c_0)$ for $ c_0\in\Ccal_0 \setminus \Ccal_0 \cap \Ccal_1$ and $ c_1\in\Ccal_1 \setminus \Ccal_0 \cap \Ccal_1$. 
    But these anima are empty because $\Ccal_0 $ and $ \Ccal_1 $ are sieves, and thus every map into these anima is an equivalence.
\end{proof}

\begin{remark}
	\Cref{cor:pushouts_of_sieves_is_union} can often be used to compute pushouts of posets.
\end{remark}


\subsection{Dwyer functors}\label{subsec:Dwyer_functors}
 
While studying the homotopy theory of $ 1 $-categories, Thomason introduced a class of fully faithful functors called \textit{Dwyer functors} \cite{MR591388}.
There, Thomason began the study of pushouts along Dwyer functors.

\begin{definition}\label{def:Dwyer}
	A functor $f \colon \Acal\to \Bcal$ is called a \defn{Dwyer functor} if the following conditions are satisfied:
	\begin{enumerate}
		\item The functor $ f $ is fully faithful, with essential image a sieve.

		\item For all $ b\in \Bcal $ such that $\Acal\crosslimits_{\Bcal}\Bcal_{/b}$ is nonempty, the \category $ \Acal\crosslimits_{\Bcal}\Bcal_{/b}$ admits a terminal object, that is, a local counit map $ fRb\to b $.
	\end{enumerate} 
\end{definition}

One of the key results regarding Dwyer functors is that (homotopy) pushouts of $ 1 $-categories along Dwyer functors remain $ 1 $-categories \cite[Theorem 1.6]{arXiv:2205.02353}. 
Using \Cref{cor:map3types,cor:sieve}, we are able to give a conceptual proof, as well as generalize this fact:

\begin{corollary}\label{cor:pushouts_along_Dwyer_functors}
    Let $\Pcal\subset \Spc$ be a full subcategory of anima containing the empty set, and consider a pushout square of \categories 
    \begin{equation}\label{square:Dwyer_pushout}
        \begin{tikzcd}[column sep=2.5em]
            \Acal \arrow[r, "g"] \arrow[d, "f"', hooked] \arrow[dr, phantom, very near end, "\ulcorner", xshift=0.25em, yshift=-0.25em] & \Ccal \arrow[d, "\fbar"] \\
            \Bcal \arrow[r, "\gbar"'] & \Dcal \comma
        \end{tikzcd}
    \end{equation}  
    where $ f $ is a Dwyer functor, and all mapping anima of $\Acal $, $ \Bcal $, and $ \Ccal $ lie in $ \Pcal $. 
    Then all mapping anima of $\Dcal$ also lie in $ \Pcal $. 
\end{corollary}

\begin{proof}
    By \Cref{cor:map3types,cor:sieve}, except for the ``cross-terms'', all mapping anima in $\Dcal$ are mapping anima in $\Bcal$ or $\Ccal$, so it suffices to examine the cross-terms. 
    So let $b\in\Bcal,c\in \Ccal$, we have to consider $|\Bcal_{b/}\crosslimits_{\Bcal} \Acal \crosslimits_{\Ccal} \Ccal_{/c}|$ and $|\Ccal_{c/}\crosslimits_{\Ccal} \Acal \crosslimits_{\Bcal} \Bcal_{/b}|$.
    
    For the former, we note that if it is not empty, then $\Acal$ being a sieve implies that $b\in\Acal$, and so the relevant anima becomes $|\Acal_{b/}\crosslimits_{\Ccal}\Ccal_{/c}|$ which is equivalent, by an elementary cofinality argument, to $\Map_{\Ccal}(g(b),c)$ and is thus in $ \Pcal $. 
    Since $\emptyset\in\Pcal$, we are done either way. 

    For the latter, using the notation from \Cref{def:Dwyer} we note that again, either this anima is empty, or by assumption,
    \begin{equation*}
    	\Acal\crosslimits_{\Bcal}\Bcal_{/b}\simeq \Acal_{/Rb} \period
    \end{equation*}
    Hence the relevant anima becomes $|\Ccal_{c/}\crosslimits_{\Ccal}\Acal_{/Rb}|$ which, again by an elementary cofinality argument, is simply $\Map_{\Ccal}(c,gRb)$.
\end{proof}

Let $ n \geq 0 $ be an integer.
Recall that \acategory $ \Ccal $ is an \defn{$ n $-category} if all mapping anima in $ \Ccal $ are $ (n-1) $-truncated.
Write $ \Cat_n \subset \Catinfty $ for the full subcategory spanned by the $ n $-categories.

\begin{corollary}
	For each $ n \geq 0 $, the inclusion $ \incto{\Cat_n}{\Catinfty} $ preserves pushouts along Dwyer functors
\end{corollary}

\begin{proof}
	Immediate from \Cref{cor:pushouts_along_Dwyer_functors} for $ \Pcal $ be the \category of $ (n-1) $-truncated anima.
\end{proof}


\subsection{Pushout products}\label{subsec:pushout_products}

\begin{corollary}
    Let $\Ccal_0\subset \Ccal $ and $ \Dcal_0\subset \Dcal $ be full subcategories. 
    Then the natural square
	\begin{equation*}
        \begin{tikzcd}[column sep=2.5em]
            \Ccal_0 \cross \Dcal_0 \arrow[r, hooked] \arrow[d, hooked] & \Ccal_0 \cross \Dcal \arrow[d, hooked] \\
             \Ccal \cross \Dcal_0 \arrow[r, hooked] &  (\Ccal\times \Dcal_0) \union (\Ccal_0\times \Dcal)
        \end{tikzcd}
    \end{equation*}  
    is a pushout square of \categories.
    Equivalently, the functor from the pushout product to the product is fully faithful. 
\end{corollary}

\begin{proof}
    Note that all functors in the square are fully faithful and the bottom right-hand \category is a full subcategory of $\Ccal \times \Dcal$.
    Thus by \cref{cor:ff-and-ff} we only have to show that map
    \begin{equation*}
    	(\Ccal\times \Dcal_0)_{(c,d_0)/}\crosslimits_{\Ccal\times\Dcal_0} (\Ccal_0\times\Dcal_0)\crosslimits_{\Ccal_0\times\Dcal} (\Ccal_0\times\Dcal)_{/(c_0,d)}
        \longrightarrow
        \Map_{\Ccal \times \Dcal}((c,d_0), (c_0, d))
    \end{equation*}
    is a weak equivalence for all $(c,d_0) \in \Ccal \times \Dcal_0$ and $(d_0,d) \in \Ccal_0 \times \Dcal$. 
    (The other case follows by symmetry.)
    We can rewrite the left \category as 
    \begin{equation*}
     	\paren{\Ccal_{c/}\crosslimits_{\Ccal} (\Ccal_0)_{/c_0}} \times \paren{(\Dcal_0)_{d_0/}\crosslimits_{\Dcal} \Dcal_{/d}}
    \end{equation*}
    which is indeed weakly equivalent to $\Map_\Ccal(c,c_0) \times \Map_\Dcal(d_0,d)$, as claimed.
\end{proof}


\subsection{Reedy categories}\label{subsec:Reedy_categories}

Classically, a \textit{Reedy $ 1 $-category} $\Rcal$ has objects in bijection with $\NN$ and the Reedy structure allows for inductive arguments on the degree, where the step from degree $n-1$ to degree $n$ is controlled by the $n$-th \textit{latching} and \textit{matching} objects.
The most fundamental instance of this is that to extend a functor $X\colon \Rcal_{\le n-1} \to \Vcal$ to $X' \colon \Rcal_{\le n} \to \Vcal$ is equivalent to specifying a factorization
\begin{equation*}
	\Lup_nX \to X'(n) \to \Mup_nX
\end{equation*}
of the canonical map $\Lup_nX \to \Mup_nX$ from the $n$-th latching to the $n$-th matching object of $ X $.
See, for example, \HTT{Corollary}{A.2.9.15}.
In this subsection we propose a notion of \textit{Reedy extension} that generalizes the situation $\Rcal_{\le n-1} \hookrightarrow \Rcal_{\le n} $.
We then use the pushout formulas discussed earlier, to show that functors out of $\Rcal_{\le n}$ admit a latching-matching description analogous to the $ 1 $-categorical situation.

\begin{definition}\label{defn:Reedy-extension}
    We say that a fully faithful functor $\Acal \hookrightarrow \Bcal$ is a \defn{Reedy extension} if there is a \defn{complementary subcategory} $\Ccal \subset \Bcal$ such that 
    \begin{enumerate}
        \item The map $\Acal^\simeq \sqcup \Ccal^\simeq \to \Bcal^\simeq$ is an equivalence.

        \item For all $c_0, c_1 \in \Ccal$ the map given by composition and inclusion
        \begin{equation*}
            |\Bcal_{c_0/} \crosslimits_{\Bcal} \Acal \crosslimits_{\Bcal} \Bcal_{/c_1}| 
            \sqcup \Map_{\Ccal}(c_0, c_1)
            \longrightarrow \Map_\Bcal(c_0, c_1)
        \end{equation*}
        is an equivalence.
    \end{enumerate}
\end{definition}

\begin{remark}
    Equivalently, we have that $\Acal \hookrightarrow \Bcal$ is a Reedy extension if it satisfies
    \begin{enumerate}
        \item For all $c \in (\Bcal^\simeq \setminus \Acal^\simeq)$ the identity $\id{c}$ does not factor through any object in $\Acal$.

        \item If two morphisms $f\colon c_0 \to c_1$ and $g\colon c_1 \to c_2$ with $c_i \in (\Bcal^\simeq \setminus \Acal^\simeq)$ each do not factor through an object in $\Acal$, then neither does their composite.

        \item For any two $c_0, c_1 \in (\Bcal^\simeq \setminus \Acal^\simeq)$ the composition map
        \begin{equation*}
            |\Bcal_{c_0/} \crosslimits_{\Bcal} \Acal \crosslimits_{\Bcal} \Bcal_{/c_1}| \too \Map_\Bcal(c_0, c_1)
        \end{equation*}
        is a monomorphism.
    \end{enumerate}
    In this case there is a well-defined and unique complementary \category $\Ccal \subset \Bcal$ that contains those objects not equivalent to objects in $\Acal$ and those morphisms that do not factor through an object of $\Acal$.
\end{remark}

The idea of Reedy extensions is that we can describe functors out of $\Bcal$ in terms of a functor out of $\Acal$ and a factorization of the map from the matching to the latching functor over $\Ccal$.
Let 
\begin{equation*} 
    i\colon \Acal \hookrightarrow \Bcal \andeq  j \colon \fromto{\Ccal}{\Bcal}
\end{equation*}
denote the inclusions.

\begin{theorem}[(Reedy extension theorem)]\label{thm:Reedy}
    Suppose we have a Reedy extension $\Acal \hookrightarrow \Bcal$ with complementary subcategory $ j \colon \fromto{\Ccal}{\Bcal} $.
    Then for any presentable \category $\Vcal$ the square
	\begin{equation*}
		\begin{tikzcd}[column sep=8em]
			\Fun(\Bcal, \Vcal) \arrow[d, "i^*"'] \arrow[r, "(j^*i_!i^* \to j^* \to j^*i_*i^*)"] & {\Fun([2], \Fun(\Ccal, \Vcal))} \arrow[d, "\ev_{0 < 2}"] \\
			\Fun(\Acal, \Vcal) \arrow[r, "(j^*i_! \to j^*i_*)"'] & \Fun(\{0<2\}, \Fun(\Ccal, \Vcal)) 
		\end{tikzcd}
	\end{equation*}
	is cartesian.
\end{theorem}

\begin{remark}
    \cref{thm:Reedy} was discovered independently by Krannich--Kupers, who give a more direct proof (based on ideas of Ayala--Mazel-Gee--Rozenblyum) in \cite[Theorem 3.8]{KrannichKupers24}.
    Notably, they also establish the functoriality of the square in \cref{thm:Reedy} with respect to certain functors of pairs $\Acal \subset \Bcal$.

    In the case that the Reedy extension arises from a Reedy ($ 1 $-)category, the statement of \cref{thm:Reedy} follows from \cite[\HTTthm{Corollary}{A.2.9.15} and \HTTthm{Remark}{A.2.9.16}]{MR2522659}.
\end{remark}

\begin{remark}
    One can show a priori that 
    \begin{equation*}
        i^*\colon \Fun(\Bcal, \Vcal) \too \Fun(\Acal, \Vcal)
    \end{equation*}
    is both a cartesian and a cocartesian fibration.
    \Cref{thm:Reedy} shows that the fiber of $i^*$ is canonically identified with the factorization category
    \begin{equation*}
        (i^*)^{-1}(F)
        \simeq \Bigg\{  \begin{tikzcd}[row sep=tiny, column sep=small]
        & G \arrow[dr, dotted] & \\
        j^*i_! F \arrow[rr, "\alpha"'] \arrow[ru, dotted] && j^*i_* F
        \end{tikzcd} \; \in \Fun(\Ccal, \Vcal)\Bigg\}
    \end{equation*}
    of the canonical map $\alpha\colon j^*i_!F \to j^*i_*F$.
    In this sense, \cref{thm:Reedy} says that $\Fun(\Bcal, \Vcal)$ is a \textit{bigluing category} as defined in the $ 1 $-categorical case in \cite[Definition 3.1]{Shulman15}.
    Thus, \cref{thm:Reedy} is an \categorical version of \cite[Theorem 5.12]{Shulman15}.
    Note that in Shulman's version, the condition is roughly that for every morphism $c_0 \to c_1$ its category of factorizations through $\Acal$ is either empty or connected, whereas for us the condition becomes that the \category of factorizations is either empty or weakly contractible.
\end{remark}

\begin{remark}[{(Reedy \categories)}]
    Based on Berger and Moerdijk's definition of generalized Reedy ($ 1 $-)categories \cite[Definition 1.1]{BergerMoerdijk10}, we propose the following definition.
    A \emph{Reedy \category} is an \category $\Ccal$ with a factorization system $(\Ccal^L, \Ccal^R)$ and a map $d\colon \Ccal^\simeq \to \NN$ such that $d$ induces \emph{conservative} functors
    \begin{equation*}
        d\colon (\Ccal^L)^{\op} \too \NN 
        \andeq
        d\colon \Ccal^R \too \NN \period
    \end{equation*}
    In this situation, we have that for all $n \in \NN$ the full inclusion 
    \begin{equation*}
        i\colon \Ccal_{\le n-1} \hookrightarrow \Ccal_{\le n}
    \end{equation*}
    is a Reedy extension with complementary \category
    \begin{equation*}
    	\Ccal_n^\simeq = \coprod_{[x] \in \uppi_0(\Ccal_{n}^\simeq) } \BAut(x) \subset \Ccal^\simeq \period
    \end{equation*}
    Moreover, the matching and latching objects can be computed more easily in this situation, as by a cofinality argument we can write left and right Kan extension along $i$ as:
    \begin{equation*}
        i_!(F)(x) 
        = \colim_{y \in (\Ccal_{/x})_{\le n}} F(y)
        \simeq \colim_{y \in (\Ccal_{/x}^R)_{\le n}} F(y)
        \andeq 
        i_*(F)(x) 
        = \colim_{y \in (\Ccal_{x/})_{\le n}} F(y)
        \simeq \colim_{y \in (\Ccal_{x/}^L)_{\le n}} F(y) \period
    \end{equation*}
    This yields an inductive description of $\Fun(\Ccal, \Vcal) = \lim_{n \in \NN^{\op}} \Fun(\Ccal_{\le n}, \Vcal)$ where in each step is given by a pullback square as in \cref{thm:Reedy}.
\end{remark}


\subsubsection{Proof of the Reedy extension theorem}

We first consider the special case where there is a functor $\pi\colon \Bcal \to [2]$ such that $\Acal = \Bcal_{02}$ and $\Ccal = \Bcal_1$.
The general case will then follow via a pushout computation in \cref{lem:Reedy-pushout}.

\begin{notation}
    Define a poset over $[2]$ by
    \begin{equation*}
    	P \coloneq \{0 < L < 1 < R < 2\} \to [2]
    \end{equation*}
    where $L,R \mapsto 1$ and $i \mapsto i$ otherwise.
    For any subset $S \subset P$ we write $\Bcal_S \coloneq S \times_{[2]} \Bcal$.
    For example, $\Bcal_{L1R} \simeq \Bcal_1 \times [2]$.
\end{notation}

\begin{lemma}\label{lem:Delta2}
    For any functor $\pi\colon \Bcal \to [2]$ and any presentable \category $\Vcal$, the square
    \begin{equation*}
		\begin{tikzcd}[column sep=8em]
			\Fun(\Bcal, \Vcal) \arrow[d, "i^*"'] \arrow[r, "(j^*i_!i^* \to j^* \to j^*i_*i^*)"] & {\Fun([2], \Fun(\Bcal_1, \Vcal))} \arrow[d, "\ev_{0 < 2}"] \\
			\Fun(\Bcal_{02}, \Vcal) \arrow[r, "(j^*i_! \to j^*i_*)"'] & \Fun(\{0<2\}, \Fun(\Bcal_1, \Vcal)) 
		\end{tikzcd}
	\end{equation*}
	is cartesian.
\end{lemma}

\begin{proof}
    We first show that there is a pushout square
    \begin{equation*}
	    \begin{tikzcd}
	        \Bcal_{LR} \arrow[r, hooked] \arrow[d, hooked] & \Bcal_{0LR2} \arrow[d, hooked] \\
	        \Bcal_{L1R} \arrow[r, hooked] & \Bcal_{0L1R2} \period
	    \end{tikzcd}
    \end{equation*}
    By \cref{cor:ff-and-ff} all we need to check is that $\Map_{\Bcal}((x,1), (y,i))$ and $\Map_{\Bcal}((y,i), (x,1))$ are computed correctly for $i=0,2$.
    (Without loss of generality it suffices to only consider the former.)
    We compute
    \begin{equation*}
        |(\Bcal_{L1R})_{(x,1)/} \times_{\Bcal_{LR}} (\Bcal_{0LR2})_{/(y,0)}| = \emptyset
        \isomorphism
        \Map_{\Bcal_{0L1R2}}((x,1), (y,0))
    \end{equation*}
    and
    \begin{align*}
        |(\Bcal_{L1R})_{(x,1)/} \times_{\Bcal_{LR}} (\Bcal_{0LR2})_{/(y,2)}| &\equivalent |\Bcal_{x/} \crosslimits_{\Bcal} \Bcal_{/y}| \equivalent \Map_{\Bcal}(x,y) \\
        &\equivalence \Map_{\Bcal_{0L1R2}}((x,1), (y,2)) \period
    \end{align*}
    so the square is a pushout.

    Mapping into $\Vcal$ we get that the right-most square in the diagram
	\begin{equation*}
		\begin{tikzcd}
			\Fun(\Bcal_{012}, \Vcal) \arrow[r, "\Lan", hooked] \arrow[d] & \Fun(\Bcal_{0L12}, \Vcal) \arrow[r, "\Ran", hooked] \arrow[d] & \Fun(\Bcal_{0L1R2}, \Vcal) \arrow[r] \arrow[d] & \Fun(\Bcal_{L1R}, \Vcal) \arrow[d] \\
			\Fun(\Bcal_{02}, \Vcal) \arrow[r, "\Lan"', hooked] & \Fun(\Bcal_{0L2}, \Vcal) \arrow[r, "\Ran"', hooked] & \Fun(\Bcal_{0LR2}, \Vcal) \arrow[r] & \Fun(\Bcal_{LR}, \Vcal)
		\end{tikzcd}
	\end{equation*}
    is cartesian. 
    The horizontal functors in the left-most square are given by left Kan extension along the full inclusions $\Bcal_{012} \hookrightarrow \Bcal_{0L12}$ and $\Bcal_{02} \hookrightarrow \Bcal_{0L2}$, respectively.
    This square is cartesian: both horizontal functors are fully faithful and by inspecting the pointwise formula for left Kan extension we see that a functor $F\colon \Bcal_{0L12} \to \Vcal$ is left Kan extended from $\Bcal_{012}$ if and only if $\restrict{F}{\Bcal_{0L2}} \colon \Bcal_{0L2} \to \Vcal$ is left Kan extended from $\Bcal_{02}$.
    The same argument shows that the middle square, in which the horizontal functors are given by right Kan extension, is cartesian.
    The claim now follows by using pullback pasting to combine the three squares to the desired pullback square.
\end{proof}

\begin{notation}\label{ntn:Reedy_extension_D}
	Let $ i \colon \Acal \hookrightarrow \Bcal$ be a Reedy extension with complementary \category $ \Ccal $.
	Let $\Dcal \subset \Bcal \times [2]$ be the (non-full) subcategory uniquely described by the properties:
	\begin{enumerate}
	    \item $\Dcal^\simeq \simeq \Acal^\simeq \times \{0 < 2\} \sqcup \Ccal^\simeq \times \{1\}$ \period 

	    \item $\Acal \times \{0 < 2\} \subset \Dcal$ is full and $\Ccal \times \{1\} \subset \Dcal$ is full.

	    \item $\Map_\Dcal( (a,0), (c,1) ) = \Map_\Bcal(a, c)$ and $\Map_\Dcal( (c,1), (a,2) ) = \Map_\Bcal(c, a)$ for all $a \in \Acal$ and $c \in \Ccal$.
	\end{enumerate}
	We denote the canonical projections by $ \pi\colon \Dcal \to [2] $ and $ p\colon \Dcal \to \Bcal $, and the full inclusion $ \Acal \times \{0 < 2\} \hookrightarrow \Dcal$ by $k$.
\end{notation}

\begin{lemma}\label{lem:Reedy-pushout}
	With \Cref{ntn:Reedy_extension_D}, the square
	\begin{equation*}
		\begin{tikzcd}
		    \Acal \times \{0 < 2\} \arrow[d, "\pr_\Acal"'] \arrow[r, "k", hooked] & \Dcal \arrow[d, "p"] \\
		    \Acal \arrow[r, "i"', hooked] & \Bcal
		\end{tikzcd}
	\end{equation*}
	is a pushout square.
\end{lemma}

\begin{proof}
    As in the proof of \cref{cor:ff-and-ff}, let let $ \Pcal $ denote the pushout and consider the natural functor $F\colon \Pcal \to \Bcal$.
    Since $p\colon \Dcal \to \Bcal$ is essentially surjective, $F\colon \Pcal \to \Bcal$ is essentially surjective.
    In order to show that $ F $ is fully faithful we use the formulas for mapping anima from \Cref{cor:map3types,prop:map4}.

    Let $a \in \Acal$ and $(c,1) \in \Ccal \times \{1\} \subset \Dcal$.
    Then we compute
    \begin{align*}
        \Map_\Pcal(i(a), p(c,1)) &\simeq |\Acal_{a/} \crosslimits_{\Acal} (\Acal \times \{0 < 2\}) \crosslimits_{\Dcal} \Dcal_{/(c,1)}| \\ 
        &\simeq |(\Acal \times \{0\})_{(a,0)/} \crosslimits_{\Dcal} \Dcal_{/(c,1)}| \\
        &\simeq \Map_\Ccal(i(a), p(c,1)) \comma
    \end{align*}
    and similarly for the mapping anima $\Map_\Pcal(p(c,1), i(a))$.
    It remains to compute the fourth mapping anima, i.e., to show that for $(c_0,1), (c_1,1) \in \Ccal \times \{1\} \subset \Dcal$ the square
    \begin{equation*}\begin{tikzcd}[column sep = large]
        \abs{\Dcal_{(c_0,1)/} \crosslimits_{\Dcal} (\Acal \times \{0 < 2\}) \crosslimits_{\Dcal} \Dcal_{/(c_1,1)}} \arrow[r] \arrow[d] &
        \Map_{\Dcal}((c_0,1), (c_1,1)) \arrow[d, "\gbar"] \\
        \abs{\Dcal_{(c_0,1)/} \crosslimits_{\Dcal} (\Acal \times \{0 < 2\}) \crosslimits_{\Acal} \Ar(\Acal) \crosslimits_{\Acal} (\Acal \times \{0 < 2\}) \crosslimits_{\Dcal} \Dcal_{/(c_1,1)}} \arrow[r] & 
        \Map_{\Bcal}(c_0, c_1)
    \end{tikzcd}\end{equation*}
    is a pushout square.
    The top left anima is empty as we cannot factor the identity $\id{1}$ through $\{0,2\} \subset [2]$.
    The top right anima is $\Map_\Ccal(c_0, c_1)$.
    The bottom left anima can be simplified as the second factor must be in $\Acal \times \{2\}$ and the forth factor must be in $\Acal \times \{0\}$.
    In summary, the condition becomes that the map
    \begin{equation*}
        \abs{\Dcal_{(c_0,1)/} \crosslimits_{\Dcal} (\Acal \times \{2\}) \crosslimits_{\Acal} \Ar(\Acal) \crosslimits_{\Acal} (\Acal \times \{0\}) \crosslimits_{\Dcal} \Dcal_{/(c_1,1)}} 
        \sqcup \Map_{\Ccal}(c_0, c_1) 
        \longrightarrow
        \Map_{\Bcal}(c_0, c_1)
    \end{equation*}
    is an equivalence.
    After further rewriting,  the left term becomes 
    $\abs{\Bcal_{c_0/} \crosslimits_\Bcal \Acal \crosslimits_\Bcal \Bcal_{/c_1}}$
    and thus the map is an equivalence exactly be the definition of Reedy extension in \cref{defn:Reedy-extension}.
\end{proof}

\begin{proof}[Proof of \Cref{thm:Reedy}]
	Applying \cref{lem:Delta2} to $ \pi\colon \Dcal \to [2] $, we see that the right square in the diagram is 
	\begin{equation*}
		\begin{tikzcd}[sep=3em]
		    \Fun(\Bcal, \Vcal) \arrow[d, "i^*"'] \arrow[r, "p^*"] &
		    \Fun(\Dcal, \Vcal) \arrow[d, "k^*"{description}] \arrow[r] &
		    \Fun(\Ccal \times [2], \Vcal) \arrow[d, "\ev_{0 < 2}"] \\
		    \Fun(\Acal, \Vcal) \arrow[r, "\pr_{\Acal}^*"'] &
		    \Fun(\Acal \times \{0 < 2\}, \Vcal) \arrow[r] &
		    \Fun(\Ccal \times \{0 < 2\},\Vcal)
		\end{tikzcd}
	\end{equation*}
	is cartesian.
	By \Cref{lem:Reedy-pushout}, the left square is also cartesian.
	Hence the large outer square is cartesian, as desired.
\end{proof}


\DeclareFieldFormat{labelnumberwidth}{#1}
\printbibliography[keyword=alph, heading=references]
\DeclareFieldFormat{labelnumberwidth}{{#1\adddot\midsentence}}
\printbibliography[heading=none, notkeyword=alph]

\end{document}